\documentclass[12pt,leqno]{amsart}
\usepackage[margin=3.5cm]{geometry}
\usepackage[utf8]{inputenc}
\usepackage{graphicx, array, url, tcolorbox, xcolor, amsmath, amsfonts, amsthm, amssymb, amstext, wrapfig}
\usepackage{float}
\usepackage{caption}
\usepackage{caption}
\usepackage{todonotes}
\usepackage{tikz}
\usetikzlibrary{decorations.pathreplacing}
\usetikzlibrary{arrows.meta}

\DeclareMathOperator{\id}{id}

\newcommand{\R}{\mathbb R}

\newcommand{\ndim}{\dim_N}

\newcommand{\asdim}{\text{asdim }}

\newcommand{\diam}{\text{diam}}

\newcommand{\lbase}{\lambda}
\newcommand{\psl}{\partial_{\mathrm{sl}}}

\newcommand{\calN}{\mathcal{N}}

\newcommand{\calU}{\mathcal{U}}

\newcommand{\FF}{\mathbb{F}}

\newcommand{\bfa}{\textbf{a}}
\newcommand{\bfb}{\textbf{b}}
\newcommand{\bfc}{\textbf{c}}

\newcommand{\bfx}{\textbf{x}} 
\newcommand{\bfy}{\textbf{y}}  

\newcommand{\gothic}{\mathfrak}

\newcommand{\go}{{\gothic o}}

\newcommand{\sD}{{\sf D}}

\newcommand{\sK}{{\sf K}}   
     
\newcommand{\sM}{{\sf M}}

\newcommand{\sQ}{{\sf Q}}

\newcommand{\kk}{{\sf k}}   

\newcommand{\nn}{{\sf n}}

\newcommand{\qq}{{\sf q}}   
\newcommand{\rr}{{\sf r}}

\renewcommand{\setminus}{{\smallsetminus}}

\newcommand{\ST}{\mathbin{\Big|}} 
\newcommand{\from}{\colon\thinspace} 

\newcommand{\bdy}{\partial} 
\newcommand{\vis}{\bdy_\infty}
\newcommand{\pka}{\bdy_\kappa}
\newcommand{\dvis}{d_{\mathrm{vis}}}
\newcommand{\dlog}{d_{\mathrm{log}}}

\definecolor{peach}{HTML}{F7965A}
\definecolor{springgreen}{HTML}{C6DC67}
\definecolor{royalblue}{HTML}{0071BC}
\definecolor{periwinkle}{HTML}{7977B8}
\definecolor{pinegreen}{HTML}{008B72}
\definecolor{olivegreen}{HTML}{3C8031}

\RequirePackage{hyperref}
\hypersetup{pdfpagemode={UseOutlines},
bookmarksopen=true,
bookmarksopenlevel=0,
hypertexnames=false,
colorlinks=true, 
citecolor=magenta, 
linkcolor=blue, 
urlcolor=teal, 
pdfstartview={FitV},
unicode,
breaklinks=true,
}

\newcounter{sec_counter} \numberwithin{sec_counter}{section}

\newtheorem{theorem}[sec_counter]{Theorem}
\newtheorem{lemma}[sec_counter]{Lemma}
\newtheorem{convention}[sec_counter]{Convention}
\newtheorem{corollary}[sec_counter]{Corollary}
\newtheorem{proposition}[sec_counter]{Proposition}
\newtheorem{introthm}{Theorem}
\newtheorem{introcor}[introthm]{Corollary}

\newtheorem*{introques}{Question}

\theoremstyle{definition}
\newtheorem{definition}[sec_counter]{Definition}
\theoremstyle{remark}
\newtheorem{remark}[sec_counter]{Remark}
\newtheorem{example}[sec_counter]{Example}

\title[Conformal gauge on the Visual Boundary of a hyperbolic group]{A new conformal gauge on the visual boundary of a hyperbolic group}
\author{Manisha Garg}
\address{University of Illinois at Urbana-Champaign, Illinois}
\email{manisha8@illinois.edu, manishagmath@gmail.com}
\date{\today}


\begin{document}

\begin{abstract}
We study a logarithmic re-gauging of the visual boundary, arising from the logarithmic  metric on the sublinear Morse boundary introduced in \cite{gjq}. For every proper geodesic hyperbolic space $X$, we prove that $d_{\mathrm{log}}(\xi,\eta)\asymp (1+(\xi\mid\eta)_o)^{-\theta}$. Consequently, under the natural identification of boundaries, $d_{\mathrm{log}}$ is bi-Lipschitz equivalent to the polynomial Floyd boundary metric associated with the radial density $g_\theta(t)=(1+t)^{-1-\theta}$. Every quasi-isometry between proper geodesic hyperbolic spaces induces bi-Lipschitz map between their logarithmic boundaries with a fixed admissible exponent.

For every non-elementary hyperbolic group, the logarithmic boundary is nondoubling and has infinite Hausdorff and Assouad dimensions. Nevertheless, the logarithmic re-gauging preserves Nagata dimension and continues to recover the asymptotic dimension of the group. It also detects one-endedness through linear connectedness (bounded turning) and preserves uniform perfectness. Thus the logarithmic and visual metrics are not  quasisymmetrically equivalent, although several coarse-geometric features remain visible in the logarithmic boundary.
\end{abstract}

\maketitle

\begingroup
  \setlength{\parskip}{0pt}  
  \setcounter{tocdepth}{1}
  \tableofcontents
\endgroup

\section{Introduction}

Boundaries at infinity capture the large-scale geometry of a metric space in the geometry of its escaping directions. For a Gromov hyperbolic space $X$, the boundary $\partial_\infty X$ comprises of equivalence classes of quasi-geodesic rays and it provides a much more manageable compact topological model that can be used to understand the geometry of $X$ at large scale \cite{gromov:essay}. This boundary, often referred to as the visual boundary or Gromov boundary, has a rich topological, metric, and measure theoretic structure.

In the case of a hyperbolic group, the topology of the boundary already detects substantial algebraic and coarse-geometric information. For example, connectedness of the boundary is closely related to the number of ends of the group and decomposability of the group, while its topological dimension detects the asymptotic dimension of the group \cite{BestMess:local_connected91, HruRuane:connectivitysurvey23, swarup96:cutpointconj,  bowditch99:cutpointgeneralization, gromov:essay, BuyaloLebedevaPontryagin2007}. Furthermore, Bowditch's deep and beautiful result characterizes the splitting of a hyperbolic group over a two-ended subgroup in terms of the local cut points of the boundary \cite{bowditch98:splitting}. More generally, Bowditch constructed the canonical JSJ decomposition of $G$ over two-ended subgroups directly from the local cut-point structure of the boundary. This boundary-theoretic construction reflects the canonicality and quasi-isometry invariance of the decomposition. The JSJ decomposition organizes the essential splittings of $G$ over two-ended subgroups and plays a central role in understanding $\operatorname{Out}(G)$, the outer automorphism group of $G$
\cite{Sela97:JSJ,MilNeuSwar99:JSJtorsion,bowditch98:splitting}.

In addition to the rich topological structure, the metric on the visual boundary further measures how rapidly two escaping directions at infinity separate from each other. This helps us to understand the geometry of the rays which lie in the interior purely in terms of the distances of corresponding boundary points. The visual boundary of a hyperbolic space carries a natural \textit{conformal gauge} of visual metrics, that is, a quasisymmetric equivalence class of metrics. More concretely, if $X$ is a proper geodesic hyperbolic space with basepoint $\go$, then a visual metric satisfies
\[
d_{\mathrm{vis}}(\bfa,\bfb)
\asymp
e^{-\varepsilon(\bfa\mid\bfb)_{\go}},
\]
where $(\bfa\mid\bfb)_{\go}$ is the Gromov product and $\varepsilon>0$ is a sufficiently small visual parameter. Thus two boundary points are close when
geodesic rays representing them travel together for a long time. For definitions, see Section~\ref{S:background}.

In the cocompact setting, Paulin showed that the quasi-isometry type of a hyperbolic space is determined by the quasi-M\"obius structure of its boundary \cite{Paulin96:qmqi}; see also \cite{Bourdon95:cat-1qc, GhysHarpe:hyp}. Bonk and Schramm subsequently developed a broader correspondence between coarse maps of visual hyperbolic spaces and geometrically controlled maps such as power quasisymmetries of their boundaries \cite{BonkSchramm:hyperbolic}. We refer the reader to the surveys \cite{survey:kapovichbenakli, survey:Kleinericm, survey:stark25} for more details on conformal and quasiconformal structures on boundaries.

The geometry of the boundary can be used both to obstruct and to construct quasi-isometric embeddings. On the obstruction side, Pansu introduced conformal dimension and used it to distinguish quasi-isometry types among negatively curved homogeneous spaces and their cocompact lattices \cite{Pansu89:lattices}. On the constructive side, Bonk and Kleiner used quantitative connectedness
of visual boundaries, called linear connectedness (also known as bounded turning), to prove that $\mathbb H^2$  admits a quasi-isometric embedding into $G$ if and only if $G$ is not virtually free. Equivalently, the visual boundary of $G$ contains a quasicircle if and only if $G$ is not virtually free \cite[Theorem~1 and Corollary~2]{BonkKleiner:llc_one_ended}; see also \cite{HruRuane:connectivitysurvey23} for a different proof. Under the stronger hypothesis that $G$ does not virtually split over a finite or two-ended subgroup, Mackay strengthened the result by showing that  every finite collection of bi-infinite geodesics in $G$ lies in the image of a quasi-isometric embedding $\mathbb H^2\to G$
\cite[Corollary~1.6]{Mackay14:quasicircles}. Thus linear connectivity of the boundary helps determine embeddability of hyperbolic planes in the space itself.

We aim to study the visual boundary of a Gromov hyperbolic group. To study a metric space, it is a common method to embed it into a well-understood model space, such as a Euclidean space, a hyperbolic space, or a product of metric trees, and then to determine the properties of our arbitrary embedded metric space based on properties of this model space. Metric notions of dimension, including asymptotic dimension, Assouad dimension, Nagata (capacity) dimension, and conformal dimension, provide important tools for understanding when such embeddings exist and what geometric information they preserve.

By Coornaert's Ahlfors-regularity theorem \cite{measure:Coornaert93}, for a non-elementary hyperbolic groups its visual boundary is doubling, or equivalently, has finite Assouad dimension. Assouad's embedding theorem therefore implies that, for every $0<\alpha<1$, the snowflaked boundary
\vspace{-0.15cm}
\[
    \bigl(\partial_\infty G,d_{\mathrm{vis}}^\alpha\bigr)
\]
admits a bi-Lipschitz embedding into some finite-dimensional Euclidean space \cite[Theorem~12.2]{Heinonen2001LecturesonAnalysis}. Thus, after snowflaking, every visual boundary admits a finite-dimensional Euclidean model space. This finite-dimensionality is a key ingredient in the work of Bonk and Schramm, who show that a Gromov hyperbolic geodesic space with bounded growth at some scale is roughly similar to a convex subset of a finite-dimensional real hyperbolic space \cite{BonkSchramm:hyperbolic}. 

Similarly, the Nagata dimension (also known as capacity dimension for compact spaces) provides a tree-theoretic embedding theory. More precisely, Lang and Schlichenmaier prove that for a metric space $Z$. if $\dim_N Z\le  n$, then for every sufficiently small exponent $p\in(0,1)$, the snowflake $(Z,d^p)$ admits a bi-Lipschitz embedding into a product of $n+1$ metric trees \cite[Theorem~1.3]{LangSchlichenmaier2005}. Buyalo and Lebedeva associated the Nagata dimension of the boundary to the asymptotic dimension of a hyperbolic group. To wit, they proved
\[
    \asdim G
    =
    \dim(\partial_\infty G)+1 = \ndim(\vis G)+1,
\]
where $\asdim, \dim, \ndim$ represent the asymptotic, topological, and Nagata dimensions, respectively, and $G$ is a Gromov hyperbolic group. They utilize aforementioned dimensions to show that $G$ admits a quasi-isometric embedding into a product of $\dim(\partial_\infty G)+1$ metric trees \cite{BuyaloLebedevaPontryagin2007}. Buyalo, Dranishnikov, and Schroeder subsequently showed that the factors may be chosen to be binary trees \cite{BuyaloDranishnikovSchroeder07}. Thus the boundary dimension governs embeddings of both the boundary and the ambient group itself into finite products of model spaces.

These examples illustrate how metric properties of the visual boundary determine algebraic, topological, geometric, and analytic information about the hyperbolic group. All of these aforementioned notions, however, depend on the visual metric of the boundary. These observations naturally lead to the following question:
\begin{introques}
    Which of these properties are intrinsic to the large-scale geometry of the interior of a hyperbolic space, and which depend on the particular exponential scaling of a visual metric of its visual boundary?
\end{introques}

In this paper, we address this question by studying a different metric, $d_{\mathrm{log}}$, introduced in \cite{gjq} on the sublinear Morse boundary $\pka X$, where $\kappa:[0, \infty) \to [1, \infty)$ is a sublinear function: it is increasing, concave and
$
\lim_{r\to\infty}~{\kappa(r)}/{r}=0.
$
It is a well-known fact that the visual boundary and the sublinear Morse boundary of proper geodesic hyperbolic space $X$ are homeomorphic:
$$\pka X \cong \vis X.$$


Another flexible source of metrics at infinity is provided by Floyd's radial distance rescaling construction. Buckley and Kokkendorff show that by using a certain nice class of functions, such as the polynomial
radial density $g_\theta(t)=(1+t)^{-1-\theta}$, the Floyd boundary and the visual boundary are homeomorphic \cite{4BK09:floydideal}. Moreover, one retrieves the class of visual metrics by utilizing exponential edge weights in the construction of Floyd metrics. Introduced by Floyd in the study of Kleinian limit sets \cite{Floyd80}, Floyd boundaries have subsequently played an important role in relative hyperbolicity. In particular, Gerasimov constructed a continuous equivariant map from a suitable Floyd boundary onto the Bowditch boundary, Gerasimov and Potyagailo analyzed the fibers of this map and used Floyd compactifications in quasi-isometry problems, and Potyagailo and Yang employed Floyd and shortcut metrics to relate Hausdorff dimension at infinity to the growth of the group \cite{6Ger12:floydmaps_rhg,7GP13:QIfloyed_rhg,GP16:similar_rhg,8PY19:hdim_rhg}.

We prove that $d_{\mathrm{log}}$ is bi-Lipschitz to the polynomial Floyd metric. This description also yields a strong quasiconformal property: every quasi-isometry between hyperbolic spaces induces a bi-Lipschitz map between their logarithmic boundaries (visual boundary with the $d_{\mathrm{log}}$ log metric). By utilizing the modulus of continuity of the identity map 
$$f: (\vis X, d_\mathrm{vis}) \longrightarrow (\pka X, \dlog),$$ 
we show that the visual boundary with this new log metric is non-doubling with infinite Hausdorff dimension. More strongly, every nondegenerate connected subset has infinite Hausdorff dimension, and thus, the logarithmic boundary contains no nonconstant rectifiable curves. In particular, the logarithmic and visual metrics belong to different conformal gauges and the small scale geometry of the boundary distorts substantially. 

By contrast, several coarse-geometric features of a group remain detectable from its logarithmic boundary. The logarithmic re-gauging preserves Nagata dimension, and therefore it continues to recover the asymptotic dimension of a hyperbolic group. It also preserves linear connectedness whenever the visual boundary is linearly connected; for hyperbolic groups, the logarithmic boundary is linearly connected precisely when the group is one-ended. Finally, uniform perfectness of the visual boundary passes to the logarithmic boundary.

More concretely, for a proper geodesic metric space $X$, the metric constructed in \cite{gjq} is defined using the symmetric $\kappa$-fellow-traveling radius $r_\go(\bfa,\bfb)$ where $\bfa, \bfb \in \pka X$. In our first main theorem, we compare $r_\go(\bfa, \bfb)$ with the boundary Gromov product $(\bfa\mid\bfb)_\go$.

\begin{introthm}\label{thm:A}
Let $X$ be a proper geodesic hyperbolic space with basepoint $\go$. There exists a constant $C\ge  1$ such that, for all distinct $\bfa,\bfb\in\partial_\infty X\cong\partial_\kappa X$,
\[
    (\bfa\mid\bfb)_\go
    \le 
    r_\go(\bfa,\bfb)
    \le 
    (\bfa\mid\bfb)_\go
    +
    C\kappa\bigl((\bfa\mid\bfb)_\go\bigr)
    +
    C.
\]
Consequently, $1+r_\go(\bfa,\bfb)\asymp 1+(\bfa\mid\bfb)_\go$ and every logarithmic metric with admissible exponent $\theta>0$ satisfies
\[
    d_{\mathrm{log}}(\bfa,\bfb)
    \asymp
    \frac{1}
    {\bigl(1+(\bfa\mid\bfb)_\go\bigr)^\theta}.
\]
\end{introthm}

One advantage of this change of metric is that the admissible range of exponents depends only on the logarithmic base used in the construction of $d_{\mathrm{log}}$, and not on the hyperbolicity constants of the spaces being compared. Combining Theorem~\ref{thm:A} with the Floyd comparison of Buckley and Kokkendorff gives Corollary~\ref{cor:log-metric-floyd}: under the natural boundary identification, $d_{\mathrm{log}}$ is bi-Lipschitz equivalent to the Floyd boundary metric associated with
$g_\theta(t)=(1+t)^{-1-\theta}$.

The logarithmic metric thus has two independent origins. It arises intrinsically from sublinear fellow traveling on the sublinear Morse boundary, while on the boundary of a hyperbolic group, it coincides up to bi-Lipschitz equivalence with a polynomial Floyd metric. This situates the metric geometry of the boundry within the existing theory of Floyd boundaries with radially rescaled metrics.

\begin{introthm}\label{thm:B}
Let $X$ and $Y$ be proper geodesic hyperbolic spaces, and let
$\Phi:X\to Y$ be a quasi-isometry. Fix $\go\in X$ and let $\go' \in Y$. If the same admissible exponent $\theta>0$ is used on both
logarithmic boundaries, then the induced map
\[
    \Phi_*:
    (\partial_\kappa X,d^X_{\go,\theta})
    \longrightarrow
    (\partial_\kappa Y,d^Y_{\go',\theta})
\]
is bi-Lipschitz.
\end{introthm}
Theorem~\ref{thm:B} follows from Theorem~\ref{thm:qi_implies_bl} and Corollary~\ref{prop:change-basepoint-log}. It further yields a metric characterization of quasiconvexity.

\begin{introcor}\label{cor:D}
Let $G$ be a word-hyperbolic group, and let $H\le  G$ be a finitely generated
word-hyperbolic subgroup. Suppose that the inclusion $i:H\hookrightarrow G$
admits a Cannon-Thurston map
$\partial i:\partial H\longrightarrow \partial G.$
Let $\theta>0$ be a common admissible exponent for the log metrics on
$\partial H$ and $\partial G$. Then the following are equivalent:
\begin{enumerate}
    \item $H$ is quasiconvex in $G$.
    \item The Cannon-Thurston map
    $
    \partial i:
    (\partial H,d^H_{e,\theta})
    \longrightarrow
    (\partial G,d^G_{e,\theta})
    $
    is bi-Lipschitz onto its image where the boundaries are equipped with the log metric.
    \item The Cannon-Thurston map $\partial i:\partial H\to\partial G$ is
    injective.
\end{enumerate}    
\end{introcor}
We next turn to the metric geometry of the logarithmic boundary. The results summarized above are collected in the following theorem.

\begin{introthm}\label{thm:E}
Let $G$ be a non-elementary word-hyperbolic group. Then the following hold. 
\begin{enumerate}
    \item The logarithmic boundary has infinite Hausdorff and Assouad
    dimensions:
    \[
    \dim_H(\partial_\kappa G,d_{\mathrm{log}})
    =
    \dim_A(\partial_\kappa G,d_{\mathrm{log}})
    =
    \infty.
    \]
    In particular, it is non-doubling and is not quasisymmetrically equivalent
    to the boundary equipped with any visual metric.
    \item Every nondegenerate connected subset
    $K\subseteq\partial_\kappa G$ satisfies
    \[
    \dim_H(K,d_{\mathrm{log}})=\infty.
    \]
    Consequently, the logarithmic boundary contains no nonconstant
    rectifiable curves. In particular, every Lipschitz map from an interval into $(\partial_\kappa G,d_{\mathrm{log}})$ is constant.

    \item Its Nagata dimension remains equal to the topological dimension:
    \[
    \dim_N(\partial_\kappa G,d_{\mathrm{log}})
    = \dim_N(\bdy_\infty G,d_{\mathrm{vis}})
    = 
    \dim(\partial G).
    \]
    As a consequence,
    \[
    \operatorname{asdim}G
    =
    \dim_N(\partial_\kappa G,d_{\mathrm{log}})+1.
    \]

    \item The logarithmic boundary is uniformly perfect. 
    \item $
    (\partial_\kappa G,d_{\mathrm{log}})
    \text{ is linearly connected}
    \quad\Longleftrightarrow\quad
    G\text{ is one-ended}.
    $
\end{enumerate}
\end{introthm}
Therefore, the visual metrics and log metrics belong to different conformal gauges on the visual boundary, that is, $(\vis X, \dvis)$ is not quasi-symmetric to $(\pka X, \dlog)$.

\begin{introcor}\label{cor:F}
Let $G$ be a non-elementary hyperbolic group and let
$n=\dim \partial G$. For every sufficiently small $p\in(0,1)$, there
exist metric trees $T_0,\ldots,T_n$ and a bi-Lipschitz embedding
$
(\partial_\kappa G,d_{\mathrm{log}}^p)
\longrightarrow
T_0\times\cdots\times T_n.
$

On the other hand, for every $p\in(0,1]$, the space $(\partial_\kappa G,d_{\mathrm{log}}^p)$ admits no bi-Lipschitz embedding
into any finite-dimensional Euclidean space.
\end{introcor}


To recapitulate, the polynomial decay substantially alters the metric structure of the boundary while preserving several coarse-geometric features of the group. On the one hand, no snowflake of the boundary admits a bi-Lipschitz embedding into a finite-dimensional Euclidean space. On the other hand, sufficiently small snowflakes still embed bi-Lipschitzly into finite products of metric trees, and linear connectedness persists. From the Floyd perspective, these results separate the `quanitative' features specific to exponential visual scaling from those that are stable under polynomial radial rescaling.

\subsection{Notation} Throughout this paper, we fix the following notation and convention.
\begin{center}
    \begin{tabular}{c|l}
    Notation & Meaning \\
    \hline
    $X, Y$ & proper geodesic Gromov hyperbolic spaces \\
    $\go$ & a fixed basepoint in $X$\\
    $\kappa$ & a sublinear (concave, increasing) function\\
    $\pka X$ & $\kappa$-Morse boundary\\
    $\vis X$ &  visual boundary\\
    $\cong$ & homeomorphism\\
    $\partial X$ & when we identify $\pka X \cong \vis X$\\
    $(\bfa \mid \bfb)_{\go}$ & Gromov product with respect to the basepoint $\go$\\ 
    $\bfa, \bfb, \bfc$ & elements of $\pka X$ or $\vis X$\\
    $d_{\go, \lambda, \theta}^{X, \kappa}$ & logarithmic metric on $\bdy X$ with basepoint $\go$,\\& base $\lambda$, admissible exponent $\theta$ and a sublinear function $\kappa$\\
    $d_{\mathrm{log}}$ & logarithmic metric on $\bdy X$ with base $2$, \\ &admissible $\theta < 1/6$, fixed basepoint $\go$ and a fixed $\kappa$\\
    $\dim_H (X)$ & Hausdorff dimension of $X$\\
    $\dim_N (X)$ & Nagata dimension of $X$\\
    $\asdim (X)$ & Asymptotic dimension of $X$\\
    $\dim_A (X)$ & Assouad dimension of $X$\\
    $A\asymp B$ & $\exists C\ge 1$ so that $\tfrac{1}{C} B\le A\le CB$
    \end{tabular}
\end{center}

\subsection*{Organization of the paper.}
    Section~\ref{S:background} recalls the necessary background on coarse maps, hyperbolic spaces, visual boundaries, sublinear Morse boundaries, visual metrics, logarithmic metrics and Floyd metrics.
    Section~\ref{S:relationship} proves Theorem~\ref{thm:A}.
    Section~\ref{S:qi-invariance} proves that quasi-isometries induce bi-Lipschitz maps between logarithmic boundaries and discusses consequences for hyperbolic groups and quasiconvex subgroups (Theorem~\ref{thm:B} and Corollary~\ref{cor:D}).
    Section~\ref{S:metric-analysis-preliminaries} collects the metric-analysis tools used later, including Hausdorff gauge measures, modulus estimates, doubling, and Nagata dimension.
    Finally, Section~\ref{S:metric-consequences} includes the metric geometry of the logarithmic boundary, proving the results on various dimensions including Hausdorff, Assouad, Nagata and asymptotic dimensions, and discusses metric properties of non-doubling, linear connectedness, and uniform perfectness. We prove Theorem~\ref{thm:E}.
  
\subsection*{Acknowledgment}
The author is grateful to Jeremy Tyson for his exceptional support and guidance, and for invaluable feedback during the preparation of this paper. The author would also like to thank Yulan Qing for many fruitful discussions and her guidance. The author further thanks Emily Stark, as well as the organizers and participants of the \href{https://sites.google.com/view/cubesinggt/home}{workshop} on cube complexes in geometric group theory, for stimulating discussions and for introducing the author to several ideas related to the themes of this paper. 

\subsection*{Technology Disclaimer}
All figures were created using the web-based tool \href{https://www.mathcha.io/}{Mathcha}. ChatGPT and the Grammarly browser extension were used solely for light language editing and
polishing in a very few places of the manuscript. No generative AI tools were used to generate mathematical ideas or proofs.

\section{Preliminaries and background}\label{S:background}

Throughout the paper, $X$ and $Y$ denote proper geodesic Gromov hyperbolic metric spaces, unless otherwise stated. When a specific hyperbolicity constant is needed, we say that $X$ is $\delta$-hyperbolic. 

\subsection{Coarse maps and quasi-geodesics}
In this subsection, we recall several coarse-geometric definitions. These definitions make sense for arbitrary metric spaces, although later they will be applied to proper geodesic hyperbolic spaces.

\begin{definition}
    Let $(X,d_X)$ and $(Y,d_Y)$ be metric spaces. For constants $\kk\ge  1$ and $\sK\ge  0$, a map $f\from X\to Y$ is called a
    $(\kk,\sK)$-\textit{quasi-isometric embedding} if, for all $x_1,x_2\in X$,
    $$
        \frac{1}{\kk}d_X(x_1,x_2)-\sK
        \le 
        d_Y(f(x_1),f(x_2))
        \le 
        \kk d_X(x_1,x_2)+\sK.
        $$
    If, in addition, every point of $Y$ lies in the $\sK$-neighbourhood of $f(X)$, then $f$ is called a $(\kk,\sK)$-\textit{quasi-isometry}. 
\end{definition}

A \textit{geodesic} in $X$ is an isometric embedding from an interval $I\subseteq\mathbb R$ into $X$. If $I=[t_1,t_2]$, we call it a \textit{geodesic segment}; if $I=[0,\infty)$, we call it a \textit{geodesic ray}.

\begin{definition}[Quasi-geodesics]
\label{Def:Quasi-Geodesic}
    Let $\qq\geq 1$ and $\sQ\geq 0$. A $(\qq,\sQ)$-\textit{quasi-geodesic} in $X$ is a continuous map $\beta:I\to X$, where $I\subseteq\mathbb R$ is a 
    interval, such that
    \[
    \frac{1}{\qq}|s-t|-\sQ
    \leq d\bigl(\beta(s),\beta(t)\bigr)
    \leq \qq|s-t|+\sQ
    \]
    for all $s,t\in I$. A quasi-geodesic is a map that is a $(q,Q)$-quasi-geodesic for some
    $q\geq1$ and $Q\geq0$.
\end{definition}

\begin{lemma}\label{Lem:QI-images-of-quasigeodesics}
Let $I\subseteq \mathbb R$ be an interval, and let $\beta\from I\to X$ be a $(\qq,\sQ)$-quasi-geodesic. Let $\Phi\from X\to Y$ be a $(\kk,\sK)$-quasi-isometry. Then
$\Phi\circ\beta\from I\to Y$ is a $(\kk\qq,\kk\sQ+\sK)$-quasi-isometric embedding. 
\end{lemma}

Since $\Phi$ need not be continuous, the map $\Phi\circ\beta$ need not be continuous. Now, since $Y$ is geodesic, one can replace $\Phi\circ\beta$ by a continuous path
which stays a uniformly bounded distance from it. By \cite[Lemma III.1.11]{book:BridsonHaefliger}, this continuous replacement may be chosen to be a
$(\kk\qq,2(\kk\qq+\kk\sQ+\sK))$-quasi-geodesic. By abuse of notation, we continue to denote this continuous replacement by $\Phi\circ\beta$.
\begin{convention}
    Unless otherwise stated, all quasi-geodesic rays in $X$ are assumed to start at $\go$.
    If $\beta\from I\to X$ is a path, we will often use the same symbol $\beta$ for
    its image $\beta(I)\subseteq X$. Thus, for $x\in X$, we write
    \[
    d(x,\beta):=\inf\{d(x,\beta(t)):t\in I\}.
    \]
    Similarly, if $\beta_1\from I_1\to X$ and $\beta_2\from I_2\to X$ are two paths,
    then $d_{\mathrm{Haus}}(\beta_1,\beta_2)$ denotes the Hausdorff distance between
    their images. More precisely,
    \[
        d_{\mathrm{Haus}}(\beta_1,\beta_2)
        :=
        \max\left\{
        \sup_{x\in \beta_1} d(x,\beta_2),
        \sup_{y\in \beta_2} d(y,\beta_1)
        \right\}.
    \]
\end{convention}

\subsection{Hyperbolic spaces and visual boundaries}

We recall the basic definitions concerning Gromov hyperbolic spaces and their
visual boundaries. Standard references are
\cite[Chapter~III.H]{book:BridsonHaefliger}, \cite{book:lohggt}, and
\cite{ghintrinsic:vaisala2005}.

Recall for a metric space $(X,d)$ with  $p, x,y\in X$, the \textit{Gromov product} of $x$ and $y$ with respect to $p$ is defined by
\[
    (x\mid y)_p
    =
    \frac{1}{2}\big(d(x,p)+d(y,p)-d(x,y)\big).
\]

\begin{definition}[$\delta$-hyperbolic space]
Let $\delta\ge  0$. A metric space $X$ is called $\delta$-\textit{hyperbolic} if for all $x,y,z,p\in X$,
\[
     (x\mid z)_p
    \ge 
    \min\{(x\mid y)_p,(y\mid z)_p\}-\delta.
\]
We say that $X$ is \textit{Gromov hyperbolic} if it is $\delta$-hyperbolic for some $\delta\ge  0$.
\end{definition}

Let $X$ be a proper geodesic Gromov hyperbolic space, and fix a basepoint $\go\in X$. Let $\mathcal{Q}_{\go}(X)$ denote the collection of continuous quasi-geodesic
rays $\beta\from [0,\infty)\to X$ with $\beta(0)=\go$. 

\begin{definition}[Visual boundary]
Two quasi-geodesic rays $\beta_1,\beta_2\in\mathcal{Q}_{\go}(X)$ are equivalent
if their images have finite Hausdorff distance, that is, $d_{\mathrm{Haus}}(\beta_1,\beta_2)<\infty$.
The \textit{visual boundary} of $X$, denoted $\vis X$, is the set of equivalence
classes of quasi-geodesic rays based at $\go$:
\[
\vis X:=\mathcal{Q}_{\go}(X)/{\sim}.
\]
\end{definition}

In a proper geodesic Gromov hyperbolic space, this definition agrees with the
usual definition of the visual boundary using geodesic rays. Indeed, by the
Morse lemma (stability of quasi-geodesics), every quasi-geodesic ray stays within finite Hausdorff distance of a geodesic ray. Thus every class in $\vis X$ has a geodesic representative.

We now recall the extension of the Gromov product to boundary points. We use the compactification $X\cup \vis X$ and write $x_i\to \bfa$ to mean that the sequence $(x_i)$ converges to $\bfa$ in this compactification.

\begin{definition}[Gromov product of boundary points]
\label{Def:boundary-gromov-product}
Let $\bfa,\bfb\in \vis X$. The \textit{Gromov product} of $\bfa$ and $\bfb$ with
respect to $\go$ is defined by
\[
    (\bfa\mid\bfb)_{\go}
    :=
    \sup
    \,
    \liminf_{i,j\to\infty} (x_i\mid y_j)_{\go}
\]
where the supremum is taken over all sequences $(x_i)$ and $(y_j)$ in $X$
converging to $\bfa$ and $\bfb$, respectively. We use the convention
$
(\bfa\mid\bfa)_{\go}=\infty
$.
\end{definition}

Although Definition~\ref{Def:boundary-gromov-product} is stated using arbitrary
sequences converging to $\bfa$ and $\bfb$, geodesic representatives can be used
to estimate the product. Let $\alpha_0$ and $\beta_0$ be geodesic rays representing $\bfa$ and $\bfb$, respectively. Then the sequences
$x_n:=\alpha_0(n)$ and $y_m:=\beta_0(m)$ converge to $\bfa$ and $\bfb$ in the
visual compactification $X\cup \vis X$. By the standard boundary product
estimate \cite[Remark~III.H.3.17(5)]{book:BridsonHaefliger}, we have
\begin{equation}
\label{eq:BH_upper_lower_bound_P}
(\bfa\mid\bfb)_{\go}-2\delta
\le 
\liminf_{n,m\to\infty}
(\alpha_0(n)\mid \beta_0(m))_{\go}
\le 
(\bfa\mid\bfb)_{\go}.
\end{equation}

We will also use the following standard estimate for Gromov products under quasi-isometries. The finite-point version is due to V\"ais\"al\"a \cite[Theorem~3.21 and discussion in \S~5.32]{ghintrinsic:vaisala2005}; the boundary version follows by passing to sequences converging to boundary points. We omit the proof.

\begin{lemma}\label{lem:QI-Gromov-product}
Let $X$ and $Y$ be proper geodesic Gromov hyperbolic spaces, and let $\Phi\from X\to Y$ be a $(\kk,\sK)$-quasi-isometric embedding. Fix $\go\in X$ and let $\go'=\Phi(\go)$. Then $\Phi$ induces a topological embedding $\Phi_*\from\vis X\longrightarrow\vis Y$.

Moreover, there exist constants $A\ge 1$ and $B\ge 0$, depending only on $\kk,\sK$ and the hyperbolicity constants of $X$ and $Y$, such that for all distinct $\bfa,\bfb\in\vis X$,
\[
\frac{1}{A}(\bfa\mid\bfb)_{\go}-B
\le 
\bigl(\Phi_*(\bfa)\mid\Phi_*(\bfb)\bigr)_{\go'}
\le 
A(\bfa\mid\bfb)_{\go}+B.
\]
If $\Phi$ is a quasi-isometry, then $\Phi_*$ is surjective and hence is a homeomorphism.
\end{lemma}

\subsection{Sublinear Morse boundaries}\label{ss:pkabdy}

We recall the definition of the sublinear Morse boundary defined in \cite{QRT1} and \cite{QRT2}. Throughout this
subsection, $X$ is a $\delta$-hyperbolic proper geodesic metric space, $\kappa$ denotes a fixed sublinear function and we fix a basepoint $\go \in X$.

\begin{definition}[Sublinear function]\label{def:kappa}
A function $\kappa\from [0,\infty)\to [1,\infty)$ is called \textit{sublinear} if
it is increasing, concave, and
\[
\lim_{r\to\infty}\frac{\kappa(r)}{r}=0.
\]
\end{definition}
The following elementary consequence of concavity will be used repeatedly.
\begin{lemma}
\label{Lem:kappa_basic}
Let $\kappa$ be a sublinear function. Then:
\begin{enumerate}
    \item For every $\lambda\ge  1$ and $t\ge  0$,
    $\kappa(\lambda t)\le  \lambda \kappa(t)$.
    \item For all $s,t\ge  0$, $\kappa(s+t)\le  \kappa(s)+\kappa(t)$.
\end{enumerate}
\end{lemma}

\begin{proof}
    We use concavity in the form $\kappa(\theta u+(1-\theta)0)
    \ge 
    \theta\kappa(u)+(1-\theta)\kappa(0)$ for $u\ge  0$ and $\theta\in[0,1]$.
    
    For the first claim, take $u=\lambda t$ and $\theta=1/\lambda$. For the second claim, the case $s=t=0$ is immediate. Otherwise, let $u=s+t$. Applying the same concavity inequality with $\theta=s/(s+t)$ and then
    with $\theta=t/(s+t)$ gives
    $$
    \kappa(s)
    \ge 
    \frac{s}{s+t}\kappa(s+t)
    +
    \frac{t}{s+t}\kappa(0), \,\,\text{ and }\,\, \kappa(t)
    \ge 
    \frac{t}{s+t}\kappa(s+t)
    +
    \frac{s}{s+t}\kappa(0).
    $$
    Adding the two inequalities yields the claim.
\end{proof}

We write $\|x\|:=d_X(\go,x)$ for $x\in X$. For a quasi-geodesic ray $\beta$, we denote $\|\beta(t)\|:=d_X(\go,\beta(t))$. For a sublinear function $\kappa$, we also use the shorthand
\[
    \kappa(\beta(t)):=\kappa(\|\beta(t)\|).
\]
For a continuous quasi-geodesic ray $\beta\from [0,\infty)\to X$ and for $r>0$, let
\[
t_r=t_r(\beta):=\min\{t\ge  0:\|\beta(t)\|=r\}.
\]
For a geodesic ray parametrized by arc length, one has $t_r=r$. We define
\[
\beta|_r:=\beta([0,t_r]).
\]

\begin{definition}[{$\kappa$-neighborhood}]  \label{Def:Neighborhood} 
For a closed set $Z$ and a constant $\nn$ define the $(\kappa, \nn)$-neighbourhood 
of $Z$ to be 
\[
\calN_\kappa(Z, \nn) = \Big\{ x \in X \ST 
  d_X(x, Z) \le   \nn \cdot \kappa(x)  \Big\}.
\]
Equivalently, a subset $A\subseteq X$ is contained in an $(\nn\kappa)$-neighborhood of $Z$ if $d_X(x,Z)\le  \nn\kappa(\|x\|)$ for every $x\in A$.
\end{definition}

Figure~\ref{fig:kappa-nbd} illustrated $(\kappa, \nn)$-neighborhood of a quasi-geodesic ray $\gamma$ in $X$.
\begin{figure}[h]
    \centering

    \tikzset{every picture/.style={line width=0.75pt}} 
    
    \begin{tikzpicture}[x=0.75pt,y=0.75pt,yscale=-1,xscale=1]
    
    \draw [color={rgb, 255:red, 67; green, 104; blue, 152 }  ,draw opacity=1 ][line width=1.5]    (125.5,2415.67) -- (538.5,2415.67) ;
    \draw [shift={(542.5,2415.67)}, rotate = 180] [fill={rgb, 255:red, 67; green, 104; blue, 152 }  ,fill opacity=1 ][line width=0.08]  [draw opacity=0] (11.61,-5.58) -- (0,0) -- (11.61,5.58) -- cycle    ;
    \draw  [color={rgb, 255:red, 0; green, 0; blue, 0 }  ,draw opacity=1 ][line width=3] [line join = round][line cap = round] (126.5,2413.67) .. controls (126.5,2414.33) and (126.5,2415) .. (126.5,2415.67) ;
    \draw [color={rgb, 255:red, 245; green, 166; blue, 35 }  ,draw opacity=1 ] [dash pattern={on 4.5pt off 4.5pt}]  (125.5,2415.67) -- (391.5,2333.67) ;
    \draw [color={rgb, 255:red, 245; green, 166; blue, 35 }  ,draw opacity=1 ] [dash pattern={on 4.5pt off 4.5pt}]  (391.5,2415.67) -- (391.5,2333.67) ;
    \draw [color={rgb, 255:red, 67; green, 104; blue, 152 }  ,draw opacity=1 ][line width=1.5]  [dash pattern={on 1.69pt off 2.76pt}]  (125.5,2415.67) .. controls (159.5,2373.67) and (246.5,2337.67) .. (391.5,2333.67) ;
    \draw [color={rgb, 255:red, 74; green, 74; blue, 74 }  ,draw opacity=1 ] [dash pattern={on 4.5pt off 4.5pt}]  (409.57,2335.67) -- (411.43,2410.67) ;
    \draw [shift={(411.5,2413.67)}, rotate = 268.59] [fill={rgb, 255:red, 74; green, 74; blue, 74 }  ,fill opacity=1 ][line width=0.08]  [draw opacity=0] (10.72,-5.15) -- (0,0) -- (10.72,5.15) -- (7.12,0) -- cycle    ;
    \draw [shift={(409.5,2332.67)}, rotate = 88.59] [fill={rgb, 255:red, 74; green, 74; blue, 74 }  ,fill opacity=1 ][line width=0.08]  [draw opacity=0] (10.72,-5.15) -- (0,0) -- (10.72,5.15) -- (7.12,0) -- cycle    ;
    
    \draw (383,2309.67) node [anchor=north west][inner sep=0.75pt]  [font=\small,color={rgb, 255:red, 74; green, 74; blue, 74 }  ,opacity=1 ]  {$x$};
    \draw (270,2369.67) node [anchor=north west][inner sep=0.75pt]  [font=\small,color={rgb, 255:red, 74; green, 74; blue, 74 }  ,opacity=1 ]  {$\Vert x\Vert $};
    \draw (116,2325) node [anchor=north west][inner sep=0.75pt]  [font=\footnotesize,color={rgb, 255:red, 67; green, 104; blue, 152 }  ,opacity=1 ] [align=left] {$\displaystyle ( \kappa ,\ \nn)$-neighborhood of $\displaystyle \gamma $};
    \draw (378,2424) node [anchor=north west][inner sep=0.75pt]  [font=\small]  {$x_{\gamma }$};
    \draw (497,2388) node [anchor=north west][inner sep=0.75pt]  [color={rgb, 255:red, 67; green, 104; blue, 152 }  ,opacity=1 ]  {$\mathbf{\gamma }$};
    \draw (418,2360.67) node [anchor=north west][inner sep=0.75pt]  [font=\small]  {$\nn\ \cdot \kappa (\Vert x\Vert )$};

    \end{tikzpicture}
    
    \caption{A $(\kappa,\nn)$-neighborhood of the geodesic ray $\gamma$, $x_\gamma$ is a nearest-point projection of $x$ onto $\gamma$, and $d(x,\gamma)\le \nn\kappa(\|x\|)$.}
    \label{fig:kappa-nbd}
\end{figure}

\begin{definition}[$\kappa$-Morse subsets]
\label{Def:Morse}
    Let $Z\subseteq X$ be a closed subset, and let $\kappa$ be a concave sublinear
    function. We say that $Z$ is \textit{$\kappa$-Morse} if there exists a proper
    function
    $
    m_Z\from [1,\infty)\times [0,\infty)\to [0,\infty)
    $
    such that the following holds.
    
    For every $(\qq,\sQ)$-quasi-geodesic segment
    $\beta\from [s,t]\to X$ with endpoints on $Z$, we have
    $$
    \beta([s,t])
    \subseteq
    \calN_{\kappa}\bigl(Z,m_Z(\qq,\sQ)\bigr).
    $$
    The function $m_Z$ is called a \textit{$\kappa$-Morse gauge} for $Z$.
\end{definition}

\begin{figure}
    \centering

\tikzset{every picture/.style={line width=0.75pt}} 

\begin{tikzpicture}[x=0.75pt,y=0.75pt,yscale=-0.8,xscale=0.8]
\path[
    fill={rgb,255:red,245;green,166;blue,35},
    fill opacity=0.18,
    draw=none
]
(77.7,5686.9)
.. controls (140,5625) and (280,5605) .. (357.7,5602.9)
.. controls (410,5600) and (455,5595) .. (490.7,5591)
-- (490.7,5772.8)
.. controls (455,5769) and (410,5765) .. (357.7,5764.1)
.. controls (280,5762) and (140,5748) .. (77.7,5686.9)
-- cycle;
\draw [color={rgb, 255:red, 205; green, 133; blue, 15 }  ,draw opacity=1 ][line width=2.25]    (77.7,5686.9) -- (485.7,5681.96) ;
\draw [shift={(490.7,5681.9)}, rotate = 179.31] [fill={rgb, 255:red, 205; green, 133; blue, 15 }  ,fill opacity=1 ][line width=0.08]  [draw opacity=0] (14.29,-6.86) -- (0,0) -- (14.29,6.86) -- cycle    ;
\draw [shift={(77.7,5686.9)}, rotate = 359.31] [color={rgb, 255:red, 205; green, 133; blue, 15 }  ,draw opacity=1 ][fill={rgb, 255:red, 205; green, 133; blue, 15 }  ,fill opacity=1 ][line width=2.25]      (0, 0) circle [x radius= 3, y radius= 3]   ;
\draw    (198.3,5684.7) .. controls (213.77,5673.09) and (286.66,5652.56) .. (294.8,5656.74) .. controls (302.93,5660.91) and (333.27,5669.09) .. (357.8,5650.7) ;
\draw [shift={(357.8,5650.7)}, rotate = 323.13] [color={rgb, 255:red, 0; green, 0; blue, 0 }  ][fill={rgb, 255:red, 0; green, 0; blue, 0 }  ][line width=0.75]      (0, 0) circle [x radius= 2, y radius= 2]   ;
\draw [shift={(198.3,5684.7)}, rotate = 323.13] [color={rgb, 255:red, 0; green, 0; blue, 0 }  ][fill={rgb, 255:red, 0; green, 0; blue, 0 }  ][line width=0.75]      (0, 0) circle [x radius= 3.35, y radius= 3.35]   ;
\draw    (357.8,5650.7) .. controls (397.8,5620.7) and (384.3,5729.7) .. (428.3,5682.7) ;
\draw [shift={(428.3,5682.7)}, rotate = 313.11] [color={rgb, 255:red, 0; green, 0; blue, 0 }  ][fill={rgb, 255:red, 0; green, 0; blue, 0 }  ][line width=0.75]      (0, 0) circle [x radius= 3.35, y radius= 3.35]   ;
\draw [color={rgb, 255:red, 139; green, 87; blue, 42 }  ,draw opacity=1 ]   (357.74,5605.9) -- (358.66,5680.9) ;
\draw [shift={(358.7,5683.9)}, rotate = 269.29] [fill={rgb, 255:red, 139; green, 87; blue, 42 }  ,fill opacity=1 ][line width=0.08]  [draw opacity=0] (10.72,-5.15) -- (0,0) -- (10.72,5.15) -- (7.12,0) -- cycle    ;
\draw [shift={(357.7,5602.9)}, rotate = 89.29] [fill={rgb, 255:red, 139; green, 87; blue, 42 }  ,fill opacity=1 ][line width=0.08]  [draw opacity=0] (10.72,-5.15) -- (0,0) -- (10.72,5.15) -- (7.12,0) -- cycle    ;

\draw (501,5661.3) node [anchor=north west][inner sep=0.75pt]  [font=\Large,color={rgb, 255:red, 139; green, 87; blue, 42 }  ,opacity=1 ]  {$\mathbf{\alpha }$};
\draw (376,5618.3) node [anchor=north west][inner sep=0.75pt]    {$\mathbf{m_{\alpha }( q,Q) \ \cdot \kappa ( r)}$};
\draw (209.7,5648.9) node [anchor=north west][inner sep=0.75pt]  [font=\large,color={rgb, 255:red, 0; green, 0; blue, 0 }  ,opacity=1 ]  {$\mathbb{\beta }$};
\draw (316.7,5626.9) node [anchor=north west][inner sep=0.75pt]  [color={rgb, 255:red, 0; green, 0; blue, 0 }  ,opacity=1 ]  {$\mathbb{\beta }( t_{r})$};

\end{tikzpicture}
\caption{The weak $\kappa$-Morse property of the geodesic ray $\alpha$. For $\beta$, a $(q,Q)$-quasi-geodesic segment with endpoints on $\alpha$, a point $\beta(t_r)$ with $\|\beta(t_r)\|=r$ satisfies $d(\beta(t_r),\alpha)\le m_\alpha(q,Q)\kappa(r)$.}
    \label{fig:placeholder}
\end{figure}
This is equivalent to the original formulation of $\kappa$-Morseness using
sublinear tracking of quasi-geodesic rays; see
\cite[Proposition~3.10]{QRT2}. A geodesic ray $\alpha\from [0,\infty)\to X$ is called \textit{$\kappa$-Morse} if its image is a $\kappa$-Morse subset of $X$.

\begin{lemma}\label{Lem:uniform-kappa-morse}
    Every geodesic ray
    in $X$ is $\kappa$-Morse. Moreover, the $\kappa$-Morse gauge can be chosen
    uniformly: for every $\qq\ge  1$ and $\sQ\ge  0$, there exists a constant
    $M=M(\delta,\qq,\sQ)$ such that if $\alpha$ is any geodesic ray and
    $\beta$ is a $(\qq,\sQ)$-quasi-geodesic segment with endpoints on $\alpha$,
    then
    $$
    \beta\subseteq N_M(\alpha)\subseteq \calN_{\kappa}(\alpha,M),
    $$ where $N_M(\alpha)$ is the ordinary $M$-neighborhood of $\alpha$.
\end{lemma}

\begin{proof}
    This is the stability of quasi-geodesics in hyperbolic spaces; see
    \cite[Theorem~III.H.1.7]{book:BridsonHaefliger}. Since $\kappa\ge  1$, the
    ordinary $M$-neighborhood $N_M(\alpha)$ is contained in the
    $M\kappa$-neighborhood $\calN_{\kappa}(\alpha,M)$.
\end{proof}

\begin{remark}[Uniform Morse gauge]
\label{Rem:uniform-morse-gauge}
    By the Morse lemma for quasi-geodesics in hyperbolic spaces
    \cite[Theorem~III.H.1.7]{book:BridsonHaefliger}, there is a constant
    $H=H(\delta,\qq,\sQ)$ such that every $(\qq,\sQ)$-quasi-geodesic segment with
    endpoints on a geodesic $\alpha$ lies in the ordinary
    $H$-neighborhood of $\alpha$. In particular, we may fix the uniform gauge 
    \begin{equation}\label{eq:uniform_gauge}  
    \sM(\qq,\sQ):=\max\{\qq,\sQ,H(\delta,3\qq,3\sQ), 3H(\delta,\qq,\sQ)+30\delta\}
    \end{equation}
    where the choice of $\qq,\sQ$ is due to standard normalization in \cite{QRT2},  $H(\delta,3\qq,3\sQ)$ provides the strong $\kappa$-Morse gauge via the weak-to-strong equivalence as in \cite[Proposition~3.10]{QRT2}, and $3H(\delta,\qq,\sQ)+30\delta$ provides the estimates needed later.
    
    Throughout the paper, every geodesic ray used as the center of a boundary neighborhood is equipped with this common strong gauge $\sM$. In particular, this is the gauge used in the definitions of $\calU_\kappa$ and of the fellow-traveling radius $r$.
\end{remark}

\begin{remark}
    We caution the reader that the terms \textit{sublinear Morse gauge} and \textit{metric conformal gauge} refer to distinct notions. For the definition of conformal gauge, see Definition~\ref{def:conf_gauge}.
\end{remark}

\begin{definition}[$\kappa$-Morse boundary]
\label{Def:kappa-boundary}
Let $\mathcal{M}^{\kappa}_\go(X)$ be the collection of $\kappa$-Morse quasi-geodesic rays based at $\go$. Two rays $\alpha,\beta\in\mathcal{M}^{\kappa}_\go(X)$ are equivalent (denoted $\alpha\sim\beta$) if they $\kappa$-fellow travel, that is, 
   if \[ \lim_{r\to\infty} \frac{d_X(\alpha_r,\beta_r)}{r} =0. \] 
The \textit{$\kappa$-Morse boundary}, is
$$
\pka X:=\mathcal{M}^{\kappa}_\go(X)/{\sim}.
$$
\end{definition}

\begin{definition}[Full sublinear Morse boundary]
If $\kappa\le  C\kappa'$ for some $C\ge  1$, there is a natural
inclusion $
\partial_\kappa X\longrightarrow\partial_{\kappa'}X$.
The family is directed, since $\kappa+\kappa'$ dominates both $\kappa$
and $\kappa'$. We define the \textit{full sublinear Morse boundary} by
\[
\psl X:=\bigcup_{\kappa}\partial_\kappa X,
\]
equipped with the topology of \cite{QRT2}, for which every
$\partial_\kappa X$ is a topological subspace.
\end{definition}

\begin{lemma}{\cite[Lemma~4.2]{QRT2}}\label{lem:qrt242}
Let $\bfb \in \pka X$. Then there exists a (not necessarily unique) geodesic ray in the class $\bfb$.
\end{lemma}

\subsubsection{Neighborhoods and topology}

Let $\alpha_0$ be a $\kappa$-Morse geodesic ray based at $\go$ with a
 fixed $\kappa$-Morse gauge $m_{\alpha_0}$ and let $\rr>0$. We define the neighborhood $\calU_\kappa(\alpha_0,\rr)$ to be the set of all classes $\bfb\in\pka X$ with the following property:
    \begin{center}
    for every $(\qq,\sQ)$-quasi-geodesic ray $\beta\in\bfb$, if
    $
    m_{\alpha_0}(\qq,\sQ)<\tfrac{\rr}{24\kappa(\rr)},$
    then
    \[
    \beta|_{\rr}
    \subseteq
    \calN_{\kappa}\bigl(\alpha_0,m_{\alpha_0}(\qq,\sQ)\bigr).
    \]
\end{center}

We say a quantity $\sD$ \textit{is small compared to a radius $\rr>0$} if $\sD \le  \tfrac{\rr}{24\kappa(\rr)}$. In original topology of \cite{QRT2}, the cut-off $r/2\kappa(r)$ is used instead. However, the choice of $24$ does not change the underlying topology of $\pka X$.

Let $\bfa\in\pka X$. By Lemma~\ref{lem:qrt242}, the class $\bfa$ contains a
geodesic ray based at $\go$. Fix one such geodesic representative and denote it
by $\alpha_0\in\bfa$. We define $\calU_\kappa(\bfa,\rr):=\calU_\kappa(\alpha_0,\rr)$. The topology on $\pka X$ is generated by the sets $\calU_\kappa(\bfa,\rr)$
for $\bfa\in\pka X,\ \rr>0$.

For more details, see \cite[\S~4]{QRT2}.

\subsubsection{Fellow traveling radius}
Next we recall the $\kappa$-fellow traveling radius from \cite{gjq}. 
Choose once and for all a based geodesic representative $\alpha_0\in \bfa\in \pka X$.

\begin{definition}[$\kappa$-fellow traveling radius]
\label{Def:kappa-fellow-traveling-radius}
Let $\bfa,\bfb\in\pka X$. Let $\alpha_0\in\bfa$ and $\beta_0\in\bfb$ be
geodesic representatives based at $\go$. Define
\begin{align*}
    R_{\beta_0}(\bfa)
    &:=
    \sup\{r>0:\bfa\in \calU_\kappa(\beta_0,r)\},
\text{ and }\\  
R_{\alpha_0}(\bfb)
    &:=
    \sup\{r>0:\bfb\in \calU_\kappa(\alpha_0,r)\}.
\end{align*}
We use the convention that the supremum of the empty set is $0$. For $\bfa\neq\bfb$, define
\begin{equation}\label{eq:def_r}  
r(\bfa,\bfb)
:=
\min\{  R_{\beta_0}(\bfa),R_{\alpha_0}(\bfb)\}.
\end{equation}

For $\bfa=\bfb$, define $r(\bfa,\bfa):=\infty$. Figure~\ref{fig:one-sided-fellow-travel} illustrates $R_{\alpha_0}(\bfb)$.
\end{definition}

\begin{figure}
    \centering
    
    \tikzset{every picture/.style={line width=0.75pt}} 
    
    \begin{tikzpicture}[x=0.75pt,y=0.75pt,yscale=-0.8,xscale=0.8]
    
    \draw [color={rgb, 255:red, 67; green, 104; blue, 152 }  ,draw opacity=1 ][line width=2.25]    (169,6186.03) -- (591.7,6185.64) ;
    \draw [shift={(595.7,6185.63)}, rotate = 179.95] [color={rgb, 255:red, 67; green, 104; blue, 152 }  ,draw opacity=1 ][line width=2.25]    (17.49,-5.26) .. controls (11.12,-2.23) and (5.29,-0.48) .. (0,0) .. controls (5.29,0.48) and (11.12,2.23) .. (17.49,5.26)   ;
    \draw [shift={(169,6186.03)}, rotate = 359.95] [color={rgb, 255:red, 67; green, 104; blue, 152 }  ,draw opacity=1 ][fill={rgb, 255:red, 67; green, 104; blue, 152 }  ,fill opacity=1 ][line width=2.25]      (0, 0) circle [x radius= 4, y radius= 4]   ;
    \draw [color={rgb, 255:red, 67; green, 104; blue, 152 }  ,draw opacity=1 ][line width=1.5]  [dash pattern={on 1.69pt off 2.76pt}]  (168,6160.63) .. controls (208,6130.63) and (481.7,6088.23) .. (523.7,6085.23) ;
    \draw [color={rgb, 255:red, 67; green, 104; blue, 152 }  ,draw opacity=1 ][line width=1.5]  [dash pattern={on 1.69pt off 2.76pt}]  (172,6220.63) .. controls (227.7,6255.37) and (504.7,6285.37) .. (566.7,6278.37) ;
    \draw [color={rgb, 255:red, 245; green, 166; blue, 35 }  ,draw opacity=1 ][line width=2.25]    (169,6186.03) .. controls (186.62,6172.82) and (231.9,6158.34) .. (279.22,6141.06) .. controls (322.85,6125.14) and (368.22,6106.84) .. (395.28,6084.97) .. controls (448.03,6042.33) and (454.96,6045.68) .. (472.8,6033.22) ;
    \draw [shift={(476.7,6030.37)}, rotate = 142.7] [fill={rgb, 255:red, 245; green, 166; blue, 35 }  ,fill opacity=1 ][line width=0.08]  [draw opacity=0] (14.29,-6.86) -- (0,0) -- (14.29,6.86) -- cycle    ;
    \draw  [draw opacity=0] (337.63,6082.15) .. controls (361.56,6114.51) and (375.7,6154.54) .. (375.7,6197.87) .. controls (375.7,6235.43) and (365.08,6270.5) .. (346.67,6300.25) -- (180.84,6197.87) -- cycle ; 
    \draw  [color={rgb, 255:red, 67; green, 104; blue, 152 }  ,draw opacity=1 ] (337.63,6082.15) .. controls (361.56,6114.51) and (375.7,6154.54) .. (375.7,6197.87) .. controls (375.7,6235.43) and (365.08,6270.5) .. (346.67,6300.25) ;  
    \draw    (484.12,6093.16) -- (487.58,6182.77) ;
    \draw [shift={(487.7,6185.77)}, rotate = 267.78] [fill={rgb, 255:red, 0; green, 0; blue, 0 }  ][line width=0.08]  [draw opacity=0] (10.72,-5.15) -- (0,0) -- (10.72,5.15) -- (7.12,0) -- cycle    ;
    \draw [shift={(484,6090.17)}, rotate = 87.78] [fill={rgb, 255:red, 0; green, 0; blue, 0 }  ][line width=0.08]  [draw opacity=0] (10.72,-5.15) -- (0,0) -- (10.72,5.15) -- (7.12,0) -- cycle    ;
    
    \draw (375,6158.77) node [anchor=north west][inner sep=0.75pt]    {$r_{\alpha _{0}}( \beta )$};
    \draw (604,6168.77) node [anchor=north west][inner sep=0.75pt]  [font=\large,color={rgb, 255:red, 67; green, 104; blue, 152 }  ,opacity=1 ]  {$\alpha _{0}$};
    \draw (482,6022.77) node [anchor=north west][inner sep=0.75pt]  [font=\large,color={rgb, 255:red, 245; green, 166; blue, 35 }  ,opacity=1 ]  {$\beta $};
    \draw (496,6025.77) node [anchor=north west][inner sep=0.75pt]   [align=left] {(q,Q)-geodesic};
    \draw (491,6124.77) node [anchor=north west][inner sep=0.75pt]  [font=\normalsize,color={rgb, 255:red, 67; green, 104; blue, 152 }  ,opacity=1 ]  {$m_{\alpha }( q,Q) \ \kappa ( r)$};

    \end{tikzpicture}

    \caption{To determine $r_{\alpha_0}(\bfb)$, for each $(q,Q)$-quasi-geodesic ray $\beta \in \bfb$, we check the largest radius $r$ so that $\beta|_{t_{r}} \subseteq \mathcal{N}_\kappa(\alpha_0, m_{\alpha_0})$ whenever $2m_{\alpha}(q,Q) \kappa(r) < r$.}
    \label{fig:one-sided-fellow-travel}
\end{figure}
\begin{convention}
If there exists $C\ge  1$ such that $C^{-1}A\le  B\le  CA$ then we write $A\asymp B$.
\end{convention}

The quantity $r(\bfa,\bfb)$ is the sublinear Morse analogue of the Gromov
product $(\bfa\mid\bfb)_{\go}$. Just as $(\bfa\mid\bfb)_{\go}$ coarsely measures
how long geodesic representatives of $\bfa$ and $\bfb$ fellow travel,
$r(\bfa,\bfb)$ measures how long the classes $\bfa$ and $\bfb$
$\kappa$-fellow travel. Thus large values of $r(\bfa,\bfb)$ mean that
$\bfa$ and $\bfb$ are close in the $\kappa$-Morse boundary topology.

In Section~\ref{S:relationship}, we make this analogy precise in
the hyperbolic setting by proving that, under the natural identification
$\vis X\cong \pka X$,
\[
1+ r(\bfa,\bfb)\asymp 1+(\bfa\mid\bfb)_{\go}.
\]

\subsection{Identifying visual and sublinear Morse boundaries}

We recall the standard identification between the visual boundary and the sublinear Morse boundary in the hyperbolic setting.

\begin{proposition}
\label{prop:visual-kappa-homeomorphism}
Let $X$ be a proper geodesic $\delta$-hyperbolic metric space with basepoint $\go$, and let $\kappa \ge 1$ be a sublinear function. Then the natural map
\[
i\from \vis X\longrightarrow \pka X
\]
defined by $i(\bfa)=[\alpha]_{\kappa}$, where $\alpha$ is any geodesic ray based at $\go$ representing $\bfa$, is a well-defined homeomorphism.
\end{proposition}

The proof is standard but we present its direct topological proof for completeness in Appendix~\ref{app:homeo}.
\begin{corollary}
\label{cor:hyperbolic-boundary-stabilizes}
Let $X$ be a proper geodesic Gromov hyperbolic space. For every sublinear function $\kappa\ge  1$, the natural identifications imply
\[
    \partial_{\mathbf 1}X
    =
    \partial_\kappa X
    =
    \psl X
    =
    \partial_\infty X
\]
canonically as topological spaces, where $\mathbf 1(t)\equiv 1$.
\end{corollary}

\begin{proof}
Every geodesic ray in $X$ is uniformly Morse, and hence is
$\mathbf 1$-Morse and $\kappa$-Morse for every $\kappa\ge  1$.
Conversely, every point of $\partial_\kappa X$ has a geodesic
representative by Lemma~\ref{lem:qrt242}. The equivalence relation on
rays is sublinear tracking and is independent of $\kappa$. Therefore
all the boundaries have the same underlying set. The assertion about
their topologies follows from
Proposition~\ref{prop:visual-kappa-homeomorphism}.
\end{proof}

\begin{remark}\label{rem:identification_vis_sublinear}
    In the remainder of the paper, we use the homeomorphism from Proposition~\ref{prop:visual-kappa-homeomorphism} to identify $\vis X$, $\pka X$ for every $\kappa$, and $\psl X$. We denote this common underlying boundary by $\partial X$. Thus, the same symbol $\bfa$ may represent either a visual boundary point or the corresponding sublinear Morse boundary point; the metric under consideration will always be specified explicitly.
\end{remark}

\subsection{Visual metrics and log metrics}
In this subsection, we recall the three metrics on $\bdy X$ - the visual metric, the log metric and the Floyd metric.

\subsubsection{Visual metric on $\bdy X$}
    We first recall the definition of a visual metric on the Gromov boundary. 
    
    \begin{definition}[Visual metric]
    \label{Def:visual-metric}
        Let $\varepsilon>0$. A metric $\dvis$ on $\bdy X$ is called a
        \textit{visual metric with parameter $\varepsilon$} if there exist constants
        $c_{\mathrm{vis}},C_{\mathrm{vis}}>0$ such that for all distinct
        $\bfa,\bfb\in\vis X$,
        \[
        c_{\mathrm{vis}}e^{-\varepsilon(\bfa\mid\bfb)_{\go}}
        \le 
        \dvis(\bfa,\bfb)
        \le 
        C_{\mathrm{vis}}e^{-\varepsilon(\bfa\mid\bfb)_{\go}}.
        \]
    \end{definition}
    
    For every proper geodesic $\delta$-hyperbolic space, visual metrics exist for all sufficiently small $\varepsilon>0$, where the allowable range of $\varepsilon$ depends on $\delta$; see
    \cite[Proposition~III.H.3.21]{book:BridsonHaefliger}. A visual metric induces the usual topology on $\bdy X$.

    \begin{remark}\label{Rem:visual-metric-parameter}
        The parameter $\varepsilon$ is part of the data of a visual metric and must be
        chosen sufficiently small in terms of the hyperbolicity constant. Changing
        $\varepsilon$ changes the metric by a snowflake-type operation. This dependence
        on a visual parameter is one of the main contrasts with the log metric defined
        below.
    \end{remark}

\subsubsection{The Log metric on $\bdy X$}

    When more than one sublinear function is under consideration, we write $r^\kappa_\go$, $\rho^\kappa_{\go,\lbase}$, and 
    $d^{X,\kappa}_{\go,\lbase,\theta}$ to display the dependence on $\kappa$. When $\kappa$ is fixed, we continue to suppress it from the notation. 
    
    Fix a sublinear function $\kappa \ge 1$ as in Section~\ref{ss:pkabdy}. We now recall the logarithmic metric construction from \cite{gjq} on stratification of sublinear Morse boundary of a proper geodesic metric space. In our context, we apply this construction to Gromov hyperbolic spaces.
    
    Let $r_\go(\bfa,\bfb)$ denote the symmetric $\kappa$-fellow-traveling radius from Definition~\ref{Def:kappa-fellow-traveling-radius}. Fix a logarithmic base $\lbase>1$. Define
    $\rho_{\go,\lbase}\from \pka X\times \pka X\to [0,1]$
    by
    \[
        \rho_{\go,\lbase}(\bfa,\bfb)
        :=
        \begin{cases}
        0, & \bfa=\bfb,\\[3pt]
        1, & \bfa\neq\bfb \text{ and } 0\le  r_\go(\bfa,\bfb)<\lbase,\\[3pt]
        \lbase^{-\lfloor \log_\lbase r_\go(\bfa,\bfb)\rfloor+1},
        & r_\go(\bfa,\bfb)\ge  \lbase.
        \end{cases}
    \]
    
    The following results are Proposition~3.9 and Proposition~3.11 from \cite{gjq}.

\begin{proposition}
\label{prop:log-quasimetric-from-gjq}
Let $\lbase>1$. The function $\rho_{\go,\lbase}$ is a quasi-metric on
$\pka X$. More precisely, there exists a constant $K=2\,\lbase\ge  1$ such that
for all $\bfa,\bfb,\bfc\in\pka X$,
\[
\rho_{\go,\lbase}(\bfa,\bfc)
\le 
K\max\{
\rho_{\go,\lbase}(\bfa,\bfb),
\rho_{\go,\lbase}(\bfb,\bfc)
\}.
\]
Moreover,
\[
\rho_{\go,\lbase}(\bfa,\bfb)
\asymp
\frac{1}{1+r_\go(\bfa,\bfb)}
\]
for all $\bfa,\bfb\in\pka X$, with constants depending only on $\lbase$.
\end{proposition}

Since $\rho_{\go,\lbase}$ is only a quasi-metric, one obtains genuine metrics by
snowflaking and applying the standard chain construction for quasi-metric
spaces; see \cite[Proposition~14.5]{Heinonen2001LecturesonAnalysis}. We use the following consequence of \cite{gjq}.

\begin{definition}[Admissible exponent]
\label{def:admissible-exponent}
Fix a logarithmic base $\lbase>1$. Define
\[
\theta_0(\lbase)
:=
\frac{\log 2}{2\log(4\lbase)}.
\]
We say that $\theta>0$ is \textit{admissible for the base $\lbase$} if $0<\theta\le  \theta_0(\lbase)$.
\end{definition}

\begin{theorem}
\label{thm:log-metric-background}
Let $\lbase>1$, and let $\theta>0$ be admissible for the base $\lbase$. Then there exists a genuine
metric $d_{\go,\lbase,\theta}^{X, \kappa}$ on $\pka X$ such that
\[
d_{\go,\lbase,\theta}^{X, \kappa}(\bfa,\bfb)
\asymp
\rho_{\go,\lbase}(\bfa,\bfb)^\theta \asymp
\frac{1}{\bigl(1+r_\go(\bfa,\bfb)\bigr)^\theta}.
\]
The metric $d_{\go,\lbase,\theta}$ induces the usual topology on the
$\kappa$-Morse boundary.
\end{theorem}

\begin{definition}[Log metric]
\label{def:log-metric}
Any metric $d_{\go,\lbase,\theta}^{X, \kappa}$ obtained as in
Theorem~\ref{thm:log-metric-background} is called a \textit{log metric} on
$\pka X$.
\end{definition}

The logarithmic metric constructed in \cite{gjq} depends on a basepoint
$\go\in X$, a logarithmic base $\lbase>1$, and a snowflaking exponent
$\theta>0$. When all parameters are displayed, we write it as
\[
d_{\go,\lbase,\theta}^{X, \kappa}.
\]

\subsubsection*{Dependence on the parameters}
By \cite[Proposition~3.16]{gjq}, changing the logarithmic base while keeping the
same snowflaking exponent changes the metric only up to bi-Lipschitz
equivalence.

Moreover, changing the snowflaking exponent changes the metric
only up to quasisymmetric equivalence; more precisely, if
$d_{\go,\lbase,\theta}^{X, \kappa}$ and $d_{\go,\lbase',\theta'}^{X, \kappa}$ are two admissible log
metrics, then the identity map
$
\id\from
(\pka X,d_{\go,\lbase,\theta}^{X, \kappa})
\longrightarrow
(\pka X,d_{\go,\lbase',\theta'}^{X, \kappa})
$
is quasisymmetric with distortion of power type.
In this paper in Section~\ref{S:relationship}, we consider the dependency on the basepoint $\go$ and on the choice of geodesic representatives in $r(\bfa, \bfb)$.

\subsubsection*{Standing convention.}
For the remainder of the paper, we fix the logarithmic base $\lambda=2$. Accordingly, $
\theta_0(2)=1/6$, and, unless stated otherwise, an exponent $\theta$ is called \textit{admissible} if $0<\theta\le 1/6$.

When all relevant data are displayed, we denote the logarithmic metric by
\[
    d^{X,\kappa}_{\go,\theta}
    :=
    d^{X,\kappa}_{\go,2,\theta}.
\]
When $\kappa$ is fixed and understood, we write $d^X_{\go,\theta}$.
When $X$, $\go$, and $\theta$ are also understood, we use the
abbreviation
\[
d_{\mathrm{log}}:=d^X_{\go,\theta}.
\]
\subsubsection{The log metric on the full sublinear Morse boundary}

Let $\kappa_0(t)\equiv 1$. By
Corollary~\ref{cor:hyperbolic-boundary-stabilizes}, we may regard
$\partial_{\kappa_0}X$ as the full sublinear Morse boundary.
\begin{definition}[Full logarithmic metric]
\label{def:full-log-metric}
For an admissible exponent $\theta>0$, define
$d^{X,\mathrm{sl}}_{\go,\lbase,\theta} := d^{X,\kappa_0}_{\go,\lbase,\theta}$ on $\psl X$.
\end{definition}
\begin{remark}
Suppose that
$\bfa\in\partial_\kappa X$ and
$\bfb\in\partial_{\kappa'}X$. Under the canonical identifications of
Corollary~\ref{cor:hyperbolic-boundary-stabilizes}, both points belong
to $\partial_{\kappa_0}X$, and their distance is
\[
d^{X,\mathrm{sl}}_{\go,\lbase,\theta}(\bfa,\bfb)
=
d^{X,\kappa_0}_{\go,\lbase,\theta}(\bfa,\bfb).
\]
Thus the value does not depend on the strata that were originally used
to describe $\bfa$ and $\bfb$.
\end{remark}

\subsubsection{Polynomial Floyd metrics}
\label{subsubsec:polynomial-floyd}

Lastly, we recall the Floyd radial deformation of a metric space, following \cite{Floyd80,4BK09:floydideal}. Fix $\theta>0$ and define
\[
    g_\theta(t)
    :=
    \frac{1}{(1+t)^{1+\theta}},
    \qquad t\geq 0.
\]
For a rectifiable path $\gamma:[0,L]\to X$, parametrized by arclength
with respect to the original metric on $X$, define its
$g_\theta$-length by
$
\operatorname{len}_{\theta}(\gamma)
:=
\int_0^L
g_\theta\bigl(d(\go,\gamma(t))\bigr)\,dt.
$
The associated path metric on $X$ is
\[
\sigma_\theta(x,y)
:=
\inf_\gamma
\operatorname{len}_{\theta}(\gamma),
\]
where the infimum is taken over all rectifiable paths in $X$ joining
$x$ to $y$.

The tail function associated with $g_\theta$ is
$
G_\theta(s)
:=
\int_s^\infty g_\theta(t)\,dt
=
\frac{1}{\theta(1+s)^\theta}.
$
In particular,
\begin{equation}
\label{eq:polynomial-floyd-tail}
G_\theta(s)
\asymp
\frac{1}{(1+s)^\theta}.
\end{equation}
Since $G_\theta(s)=\theta^{-1}(1+s)g_\theta(s)$, the function $g_\theta$ is also a sphericalization function in the sense of \cite[Section~1.6]{4BK09:floydideal}.

One can verify that $g_\theta$ is a Floyd function in the sense of \cite[Section~1.6]{4BK09:floydideal} with a Floyd constant depending only on $\theta$. It follows that $\sigma_\theta$ is a bounded metric on $X$; see \cite[Section~1.6]{4BK09:floydideal}.

Let $\overline X^{\,\theta}$ denote the metric completion of $(X,\sigma_\theta)$ and define
\[
\partial_{F,\theta}X
:=
\overline X^{\,\theta}\setminus X.
\]
We continue to denote the extension of $\sigma_\theta$ to
$\overline X^{\,\theta}$ by $\sigma_\theta$. The metric $ \sigma_\theta|_{\partial_{F,\theta}X}$ is called the \textit{polynomial Floyd boundary metric associated with $g_\theta$}, and $\partial_{F,\theta}X$ is the corresponding \textit{polynomial Floyd boundary}.

When $X$ is a Cayley graph endowed with unit edge lengths, this construction is bi-Lipschitz equivalent to the usual discrete Floyd rescaling with edge weight $f_\theta(n)=(1+n)^{-1-\theta}$. To see this, an edge whose nearest endpoint has distance $n$ from $\go$ has $g_\theta$-length between $g_\theta(n+1)$ and $g_\theta(n)$, and note that these two quantities are comparable $g_\theta(n+1)\asymp g_\theta(n)$ Moreover, $\sum_{k=n}^\infty f_\theta(k)\asymp(1+n)^{-\theta}$. Thus, the radial edge weights have exponent $1+\theta$, while the induced scaling of the boundary inherits the exponent $\theta$.

For distinct $a,b\in\partial_{F,\theta}X$, define
\[
\langle a\mid b\rangle_{\go,F,\theta}
:=
\inf
\left\{
\liminf_{i,j\to\infty}
(x_i\mid y_j)_\go
:
x_i\overset{\sigma_\theta}{\longrightarrow}a,\ 
y_j\overset{\sigma_\theta}{\longrightarrow}b
\right\}.
\]
We use angle brackets to distinguish this convention from the boundary
Gromov product of Definition~\ref{Def:boundary-gromov-product}, which
is defined using a supremum. If $X$ is proper, geodesic, and
$\delta$-hyperbolic, then the natural map
$J_\theta:\partial X\to\partial_{F,\theta}X$ is a homeomorphism and
\begin{equation}\label{eq:floyd_bk_comparsion}
\sigma_\theta(a,b)
\asymp
G_\theta\bigl(\langle a\mid b\rangle_{\go,F,\theta}\bigr)
\asymp
\bigl(1+\langle a\mid b\rangle_{\go,F,\theta}\bigr)^{-\theta};    
\end{equation}
see \cite[Corollary~2.3 and Theorem~2.11]{4BK09:floydideal}. The
comparison constants depend only on $\delta$ and $\theta$.

\subsection{Geometric maps}

In this subsection, we recall two classes of maps that have been used to study conformal structures on visual boundaries of hyperbolic groups. These notions arise naturally in analysis on metric spaces, and we record an elementary relationship between them.

Let $(X,d_X)$ and $(Y,d_Y)$ be metric spaces, and let
$\eta\colon[0,\infty)\to[0,\infty)$ be a homeomorphism. A homeomorphism
$f\colon X\to Y$ is called \textit{$\eta$-quasisymmetric} if, for every
triple of pairwise distinct points $x,y,z\in X$,
\begin{equation}\label{eq:quasisymmetry}
 \frac{d_Y(f(x),f(y))}
      {d_Y(f(x),f(z))}
 \le 
 \eta\left(
 \frac{d_X(x,y)}
      {d_X(x,z)}
 \right).
\end{equation}
The map $f$ is called \textit{quasisymmetric} if it is $\eta$-quasisymmetric for some such homeomorphism $\eta$. Two metric
spaces are \textit{quasisymmetrically equivalent} if there exists a quasisymmetric homeomorphism between them.

\begin{definition}[Conformal gauge]\label{def:conf_gauge}
Let $Z$ be a set and let $d$ be a metric on $Z$. The \textit{conformal gauge} of $d$ is the collection of metrics $d'$ on $Z$ for which the identity map
$$ \id:(Z,d)\longrightarrow (Z,d') $$
is quasisymmetric. Two metrics on $Z$ are said to belong to the same conformal gauge if the identity map between them is quasisymmetric.
\end{definition}

Recall for $L\ge 1$, a map $f\colon X\to Y$ is called an
$L$-\textit{bi-Lipschitz embedding} if
\[
 L^{-1}d_X(x,x')
 \le 
 d_Y\bigl(f(x),f(x')\bigr)
 \le 
 Ld_X(x,x')
\]
for every $x,x'\in X$. If, in addition, $f$ is surjective, then it is
called an $L$-\textit{bi-Lipschitz homeomorphism}.

Quasisymmetry is a purely metric condition introduced by
Tukia and V\"ais\"al\"a \cite{TukiaVaisala1980}; see also
\cite[Chapters~10-12]{Heinonen2001LecturesonAnalysis} for further
background. Every $L$-bi-Lipschitz homeomorphism is quasisymmetric,
with distortion function $\eta(t)=L^2t$. Moreover, if $f$ is
$\eta_1$-quasisymmetric and $g$ is $\eta_2$-quasisymmetric, then
$g\circ f$ is $(\eta_2\circ\eta_1)$-quasisymmetric. As a basic
non-bi-Lipschitz example, for every $p>0$, the map
$f\colon[0,\infty)\to[0,\infty)$ given by $f(x)=x^p$ is
quasisymmetric.

\section{Relationship between the log metric and the visual metric} \label{S:relationship}

Throughout this section, $X$ is a proper geodesic $\delta$-hyperbolic space
with basepoint $\go$. By Remark~\ref{rem:identification_vis_sublinear}, we identify $\pka X$ and $\vis X$, and simply denote them as $\bdy X$ henceforth. 

The purpose of this section is to compare the two quantities that measure
fellow traveling of boundary points. On the visual boundary, this role is played
by the Gromov product $(\bfa\mid\bfb)_\go$. On the sublinear Morse boundary,
the analogous quantity is the symmetric $\kappa$-fellow-traveling radius
$r_\go(\bfa,\bfb)$ as in Definition~\ref{Def:kappa-fellow-traveling-radius}.

\subsection{Preliminary hyperbolic estimates}
\label{subsec:preliminary-hyperbolic-estimates}

We begin with elementary estimates in hyperbolic spaces needed in subsequent proofs. 

\begin{lemma}
\label{lem:GP_basic}
Let $\bfa,\bfb\in \bdy X$ be distinct boundary points, and let
$\alpha_0,\beta_0$ be geodesic representatives of $\bfa,\bfb$, respectively. Then for all $s,t\ge  0$,
\[
(\alpha_0(s)\mid\beta_0(t))_\go
\ge 
\min\{s,t,(\bfa\mid\bfb)_\go\}-5\delta.
\]
Moreover, $ d\bigl(\alpha_0(s),\beta_0(t)\bigr) \le  s+t+10\delta - 2\min\{s,t,(\bfa\mid\bfb)_\go\}$.
\end{lemma}

\begin{proof}
The second inequality follows immediately from the first, since
\[
d\bigl(\alpha_0(s),\beta_0(t)\bigr)
=
s+t-2(\alpha_0(s)\mid\beta_0(t))_\go.
\]

We prove the first inequality. Let $P=(\bfa\mid\bfb)_\go$, and  \eqref{eq:BH_upper_lower_bound_P} gives,
\[
P-2\delta\,\,
\le \,\,
\liminf_{u,v\to\infty}
(\alpha_0(v)\mid\beta_0(u))_\go \,\,\le  \,\, P.
\]
Fix $s,t\ge  0$. For every $v\ge  s$ and $u\ge  t$, applying
$\delta$-hyperbolicity twice gives
\[
\begin{aligned}
(\alpha_0(s)\mid\beta_0(t))_\go
&\ge 
\min\{
(\alpha_0(s)\mid\alpha_0(v))_\go,
(\alpha_0(v)\mid\beta_0(u))_\go,
(\beta_0(u)\mid\beta_0(t))_\go
\}
-2\delta .
\end{aligned}
\]
Since $\alpha_0$ and $\beta_0$ are geodesic rays based at $\go$, we have
\[
(\alpha_0(s)\mid\alpha_0(v))_\go=s,
\qquad
(\beta_0(u)\mid\beta_0(t))_\go=t.
\]
Therefore, we obtain, 
\[
(\alpha_0(s)\mid\beta_0(t))_\go
\ge 
\min\{s,t,(\alpha_0(v)\mid\beta_0(u))_\go\}
-2\delta .
\]
Let $\varepsilon>0$. By the lower bound on the liminf, there exist arbitrarily
large $u$ and $v$ such that
\[
(\alpha_0(v)\mid\beta_0(u))_\go
\ge 
P-2\delta-\varepsilon.
\]
For such $u$ and $v$, we further obtain
\[
    (\alpha_0(s)\mid\beta_0(t))_\go
    \ge 
    \min\{s,t,P-2\delta-\varepsilon\}
    -2\delta
    \ge 
    \min\{s,t,P\}-4\delta-\varepsilon.
\]
Letting $\varepsilon\to 0$ gives 
\(
    (\alpha_0(s)\mid\beta_0(t))_\go
    \ge 
    \min\{s,t,P\}-4\delta.
\)
Hence, the stated estimate with $5\delta$ holds.
\end{proof}

The next estimate gives a converse type statement: if a point far out on one
geodesic ray lies close to the other ray, then the boundary Gromov product must
be large.

\begin{lemma}
\label{lem:gp_basic_2}
Let $\bfa,\bfb\in \bdy X$ be distinct boundary points, and let
$\alpha_0,\beta_0$ be geodesic representatives of $\bfa,\bfb$, respectively. Suppose that for some $S>0$ and $0 \le L \le  S$, $d\bigl(\beta_0(S),\alpha_0\bigr)\le  L$. Then
$
(\bfa\mid\bfb)_\go
\ge 
S-L-2\delta.
$
\end{lemma}

\begin{proof}
Let $P=(\bfa\mid\bfb)_\go$. Choose $t\ge  0$ such that $d\bigl(\beta_0(S),\alpha_0(t)\bigr)\le  L$. Since $\alpha_0$ and $\beta_0$ are geodesic rays based at $\go$, this implies $|t-S|\le  L$. Moreover,
\[
    \begin{aligned}
    (\alpha_0(t)\mid\beta_0(S))_\go
    &=
    \frac{1}{2}\bigl(t+S-d(\alpha_0(t),\beta_0(S))\bigr) \\
    &\ge 
    \frac{1}{2}(t+S-L) \\
    &\ge 
    S-L.
    \end{aligned}
\]
Now let $u\ge  t$ and $v\ge  S$. Applying $\delta$-hyperbolicity twice gives
\[
    \begin{aligned}
    (\alpha_0(u)\mid\beta_0(v))_\go
    &\ge 
    \min\{
    (\alpha_0(u)\mid\alpha_0(t))_\go,
    (\alpha_0(t)\mid\beta_0(S))_\go,
    (\beta_0(S)\mid\beta_0(v))_\go
    \}
    -2\delta \\
    &\ge 
    \min\{t,S-L,S\}-2\delta.
    \end{aligned}
\]
Since $t\ge  S-L$, we get
\[
(\alpha_0(u)\mid\beta_0(v))_\go
\ge 
S-L-2\delta
\]
for all $u\ge  t$ and $v\ge  S$. Hence,
\[
\liminf_{u,v\to\infty}
(\alpha_0(u)\mid\beta_0(v))_\go
\ge 
S-L-2\delta.
\]
By \eqref{eq:BH_upper_lower_bound_P}, the liminf above
is bounded above by $P$, which gives the desired inequality.
\end{proof}

\subsection{Comparing $r_\go$ with the Gromov product}
\label{subsec:comparing-r-gromov-product}

We now compare the $\kappa$-fellow-traveling radius
$r_\go(\bfa,\bfb)$ with the Gromov product $(\bfa\mid\bfb)_\go$.
Throughout this subsection, we use the uniform $\kappa$-Morse gauge
$\sM$ for geodesic rays in $X$ from Remark~\ref{Rem:uniform-morse-gauge}.
By increasing $\sM$ if necessary, we assume that for every
$(\qq,\sQ)$ there exists a constant $H=H(\delta,\qq,\sQ)$ such that whenever
a $(\qq,\sQ)$-quasi-geodesic ray and a geodesic ray represent the same point
of $\bdy X$, the quasi-geodesic ray lies in the ordinary $H$-neighborhood of
the geodesic ray, and
\[
2H(\delta,\qq,\sQ)+10\delta
\le 
\sM(\qq,\sQ).
\]
This enlargement does not affect the topology of the $\kappa$-Morse boundary.

\begin{proposition}
\label{prop:relationship_p_r}
Let $\bfa,\bfb\in \bdy X$ be distinct boundary points. Then there
exists a constant $C=C(\delta,\kappa)\ge  1$ such that
\[
(\bfa\mid\bfb)_\go
\le 
r_\go(\bfa,\bfb)
\le 
(\bfa\mid\bfb)_\go
+
C\kappa\bigl((\bfa\mid\bfb)_\go\bigr)
+
C.
\]
Consequently,
\[
\frac{r_\go(\bfa,\bfb)}{(\bfa\mid\bfb)_\go}
\longrightarrow 1
\qquad
\text{as }(\bfa\mid\bfb)_\go\to\infty.
\]
\end{proposition}

\begin{proof}
Let $P=(\bfa\mid\bfb)_\go$ and $r=r_\go(\bfa,\bfb)$. Let $\alpha_0,\beta_0$ be geodesic representatives of $\bfa,\bfb$,
respectively, based at $\go$.

We first prove the lower bound $P\le  r$. If $P=0$, there is nothing to prove.
Assume $P>0$, and let $S$ satisfy $0<S\le  P$. We show that $
\bfb\in \calU_\kappa(\bfa,S)$. Indeed let $\beta\in\bfb$ be a $(\qq,\sQ)$-quasi-geodesic ray such that $ \sM(\qq,\sQ)<\tfrac{S}{24\kappa(S)}$.  By Definition~\ref{Def:kappa-boundary-neighborhoods}, it is enough to prove
\(
    \beta|_S\subseteq \calN_\kappa(\alpha_0,\sM(\qq,\sQ)).
\)

Let $z\in \beta|_S$. Then $\|z\|\le  S$. Since $\beta$ and $\beta_0$ represent
the same boundary point, the Morse lemma gives a point $\beta_0(t')$ such that
\[
d(z,\beta_0(t'))\le  H(\delta,\qq,\sQ)
\]
as shown in Figure~\ref{fig:prop33}. 
\begin{figure}
    \centering
    \tikzset{every picture/.style={line width=0.75pt}} 
    
    \tikzset{every picture/.style={line width=0.75pt}} 

    \begin{tikzpicture}[x=0.75pt,y=0.75pt,yscale=-0.7,xscale=0.7]
    
    \draw    (73.58,203.25) -- (528.58,201.26) ;
    \draw [shift={(531.58,201.25)}, rotate = 179.75] [fill={rgb, 255:red, 0; green, 0; blue, 0 }  ][line width=0.08]  [draw opacity=0] (8.93,-4.29) -- (0,0) -- (8.93,4.29) -- cycle    ;
    \draw [shift={(73.58,203.25)}, rotate = 359.75] [color={rgb, 255:red, 0; green, 0; blue, 0 }  ][fill={rgb, 255:red, 0; green, 0; blue, 0 }  ][line width=0.75]      (0, 0) circle [x radius= 3.35, y radius= 3.35]   ;
    \draw  [draw opacity=0][dash pattern={on 4.5pt off 4.5pt}] (319.56,62.63) .. controls (357.31,97.26) and (380.98,146.99) .. (380.98,202.25) .. controls (380.98,252.78) and (361.19,298.69) .. (328.94,332.65) -- (191.58,202.25) -- cycle ; \draw  [color={rgb, 255:red, 245; green, 166; blue, 35 }  ,draw opacity=1 ][dash pattern={on 4.5pt off 4.5pt}] (319.56,62.63) .. controls (357.31,97.26) and (380.98,146.99) .. (380.98,202.25) .. controls (380.98,252.78) and (361.19,298.69) .. (328.94,332.65) ;  
    \draw [color={rgb, 255:red, 67; green, 104; blue, 152 }  ,draw opacity=1 ]   (73.58,203.25) .. controls (113.58,173.25) and (101.58,145.25) .. (178.58,121.25) ;
    \draw [color={rgb, 255:red, 67; green, 104; blue, 152 }  ,draw opacity=1 ]   (178.58,121.25) .. controls (218.58,91.25) and (230.58,146.25) .. (307.58,122.25) ;
    \draw [shift={(178.58,121.25)}, rotate = 323.13] [color={rgb, 255:red, 67; green, 104; blue, 152 }  ,draw opacity=1 ][fill={rgb, 255:red, 67; green, 104; blue, 152 }  ,fill opacity=1 ][line width=0.75]      (0, 0) circle [x radius= 3.35, y radius= 3.35]   ;
    \draw [color={rgb, 255:red, 67; green, 104; blue, 152 }  ,draw opacity=1 ]   (307.58,122.25) .. controls (318.07,114.38) and (329.57,140.12) .. (343.58,138.25) .. controls (357.6,136.38) and (365.51,117.25) .. (393.58,114.25) .. controls (421.66,111.25) and (432.34,141.39) .. (447.58,141.25) .. controls (462.29,141.12) and (493.31,146.83) .. (518.83,142.73) ;
    \draw [shift={(521.58,142.25)}, rotate = 169.11] [fill={rgb, 255:red, 67; green, 104; blue, 152 }  ,fill opacity=1 ][line width=0.08]  [draw opacity=0] (8.93,-4.29) -- (0,0) -- (8.93,4.29) -- cycle    ;
    \draw    (178.58,121.25) -- (177.58,202.25) ;
    \draw   (177.58,191) -- (187.17,191) -- (187.17,202.25) -- (177.58,202.25) -- cycle ;
    \draw [color={rgb, 255:red, 139; green, 87; blue, 42 }  ,draw opacity=1 ]   (73.58,203.25) -- (409.82,346.08) ;
    \draw [shift={(412.58,347.25)}, rotate = 203.01] [fill={rgb, 255:red, 139; green, 87; blue, 42 }  ,fill opacity=1 ][line width=0.08]  [draw opacity=0] (8.93,-4.29) -- (0,0) -- (8.93,4.29) -- cycle    ;
    \draw    (177.58,202.25) -- (162.58,240.25) ;
    
    \draw (167.58,92.25) node [anchor=north west][inner sep=0.75pt]  [color={rgb, 255:red, 67; green, 104; blue, 152 }  ,opacity=1 ]  {$z$};
    \draw (523,126) node [anchor=north west][inner sep=0.75pt]  [color={rgb, 255:red, 67; green, 104; blue, 152 }  ,opacity=1 ]  {$\beta $};
    \draw (534,189) node [anchor=north west][inner sep=0.75pt]    {$\beta _{0}$};
    \draw (380,178) node [anchor=north west][inner sep=0.75pt]    {$S$};
    \draw (177.58,202.25) node [anchor=north west][inner sep=0.75pt]    {$\beta _{0}( t')$};
    \draw (182,147) node [anchor=north west][inner sep=0.75pt]  [font=\footnotesize]  {$H( \delta ,\ q,\ Q)$};
    \draw (419,337) node [anchor=north west][inner sep=0.75pt]  [color={rgb, 255:red, 139; green, 87; blue, 42 }  ,opacity=1 ]  {$\alpha _{0}$};
    \draw (134.58,238.25) node [anchor=north west][inner sep=0.75pt]  [color={rgb, 255:red, 139; green, 87; blue, 42 }  ,opacity=1 ]  {$\alpha _{0}( s)$};

    \end{tikzpicture}

    \caption{$\beta \sim \beta_0$ and $d(z, \beta_0(t')) \le H(\delta, \qq, \sQ)$.}
    \label{fig:prop33}
\end{figure}
Thus,
\[
t'
=
d(\go,\beta_0(t'))
\le 
d(\go,z)+H(\delta,\qq,\sQ)
\le 
S+H(\delta,\qq,\sQ)
\le 
P+H(\delta,\qq,\sQ).
\]
Denoting $s=\min\{t',P\}$, by Lemma~\ref{lem:GP_basic},
\[
\begin{aligned}
d(\beta_0(t'),\alpha_0(s))
&\le 
s+t'+10\delta-2\min\{s,t',P\} \\
&=
t'-s+10\delta \\
&=
\max\{0,t'-P\}+10\delta \\
&\le 
H(\delta,\qq,\sQ)+10\delta.
\end{aligned}
\]
Therefore,
\[
\begin{aligned}
d(z,\alpha_0)
&\le 
d(z,\beta_0(t'))+d(\beta_0(t'),\alpha_0(s)) \\
&\le 
2H(\delta,\qq,\sQ)+10\delta \\
&\le 
\sM(\qq,\sQ) \\
&\le 
\sM(\qq,\sQ)\kappa(\|z\|),
\end{aligned}
\]
where the last inequality uses $\kappa\ge  1$. Hence, $z\in \calN_\kappa(\alpha_0,\sM(\qq,\sQ))$. 
Since $z\in \beta|_S$ was arbitrary, we obtain
\[
\beta|_S\subseteq \calN_\kappa(\alpha_0,\sM(\qq,\sQ)).
\]
Thus, $\bfb\in \calU_\kappa(\bfa,S)$. By the same argument with $\bfa$ and
$\bfb$ interchanged, $\bfa\in \calU_\kappa(\bfb,S)$. Therefore
$r_\go(\bfa,\bfb)\ge  S$.
Since this holds for every $0<S\le  P$, we get
\[
r_\go(\bfa,\bfb)\ge  P.
\]

We now prove the upper bound. We set $M(\delta):=\sM(1,0)$. Since $\kappa$ is sublinear, we have $\tfrac{S}{\kappa(S)}\to\infty$ as $S\to\infty$.
Hence there exists $S_0=S_0(M,\kappa)>0$ such that
\[
    M<\frac{S}{24\kappa(S)}
    \qquad
    \text{for all }S\ge  S_0.
\]
If $r\le  S_0+1$, then the desired upper bound follows after increasing the
additive constant. Thus assume $r>S_0+1$. Let $S$ satisfy $S_0<S\le r$. Since $S\le r$, by the definition of $r$ there exists $T>S$ such that
$
\bfb\in \calU_\kappa(\bfa,T).
$
Since $T>S>S_0$, we have $M=\sM(1,0)<\tfrac{T}{24\kappa(T)}$. The geodesic ray $\beta_0$ is a $(1,0)$-quasi-geodesic representative of $\bfb$. Therefore the definition of $\calU_\kappa(\bfa,T)$ gives
\[
\beta_0|_T\subseteq \calN_\kappa(\alpha_0,M).
\]
Since $S<T$, the point $\beta_0(S)$ lies in $\beta_0|_T$. Hence,
\[
d(\beta_0(S),\alpha_0)\le  M\kappa(S).
\]
Applying Lemma~\ref{lem:gp_basic_2} with $L=M\kappa(S)$ gives
\[
P\ge  S-M\kappa(S)-2\delta \quad\implies\quad S\le  P+M\kappa(S)+2\delta.
\]
If $r=\infty$, then this inequality holds for arbitrarily large $S$. This is impossible, because $S-M\kappa(S)-2\delta\to\infty$ as $S\to\infty$, while $P<\infty$. Hence $r<\infty$. Letting $S\nearrow r$, we obtain
\[
    r\le  P+M\kappa(r)+2\delta.
\]

We now remove the occurrence of $\kappa(r)$ from the right-hand side. Since $\kappa$ is sublinear, there exists $R_0=R_0(M,\delta,\kappa)>0$ such that
\[
M\kappa(t)+2\delta\le  \frac{t}{2}
\qquad
\text{for all }t\ge  R_0.
\]
If $r\ge  R_0$, then $r
\le 
P+M\kappa(r)+2\delta
\le 
P+\frac{r}{2}$. Thus $r\le  2P$. Since $\kappa$ is increasing, Lemma~\ref{Lem:kappa_basic} gives
\[
\kappa(r)
\le 
\kappa(2P)
\le 
2\kappa(P).
\]
Substituting this into $r\le  P+M\kappa(r)+2\delta
$ yields
\[
    r\le  P+2M\kappa(P)+2\delta.
\]
If $r<R_0$, the desired inequality follows after increasing the additive constant. Therefore there exists $C=C(\delta,\kappa)\ge  1$ such that
\[
r
\le 
P+C\kappa(P)+C.
\]
Combining this with the lower bound $P\le  r$, we obtain
\[
P
\le 
r
\le 
P+C\kappa(P)+C.
\]
Finally, since $\kappa(P)/P\to 0$ as $P\to\infty$, the estimate above implies the desired limit.
\end{proof}

\subsection{Consequences for logarithmic metrics}
\label{subsec:log-metric-gromov-product}

For an admissible exponent $\theta>0$, recall that
\[
d_{\mathrm{log}}:=d^X_{\go,\theta}
\]
denotes the logarithmic metric on $\partial X$ associated with the fixed sublinear function $\kappa$, basepoint $\go$, and logarithmic base $2$.
We now use Proposition~\ref{prop:relationship_p_r} to express $d_{\mathrm{log}}$ in terms of the Gromov product of boundary points.

\begin{corollary}
\label{cor:one-plus-r-one-plus-gromov-product}
Let $\bfa,\bfb\in \bdy X$ be distinct boundary points. Then
\[
1+r_\go(\bfa,\bfb)
\asymp
1+(\bfa\mid\bfb)_\go,
\]
where the comparison constants depend only on $\delta$ and $\kappa$.
\end{corollary}

\begin{proof}
Let $P=(\bfa\mid\bfb)_\go$ and $r=r_\go(\bfa,\bfb)$. By Proposition~\ref{prop:relationship_p_r},
$P\le  r\le  P+C\kappa(P)+C$ for some constant $C=C(\delta,\kappa)\ge  1$. The lower bound immediately gives $1+P\le  1+r$. For the upper bound, since $\kappa$ is increasing and sublinear, the function $t\mapsto \kappa(t)/(1+t)$ is bounded on $[0,\infty)$. Hence there exists $L=L(\kappa)>0$ such that $\kappa(P)\le  L(1+P)$ for all $P\ge  0$. Therefore,
\[
1+r
\le 
1+P+C\kappa(P)+C
\le 
1+P+CL(1+P)+C.
\]
After increasing the multiplicative constant, this gives $1+r\le  C'(1+P)$ for some  constant $C'=C'(\delta,\kappa)\ge  1$. Hence,
$
1+r_\go(\bfa,\bfb)
\asymp
1+(\bfa\mid\bfb)_\go.
$
\end{proof}

We can now express a log metric directly in terms of the Gromov product.
\begin{figure}
    \centering

\tikzset{every picture/.style={line width=0.75pt}} 

\begin{tikzpicture}[x=0.75pt,y=0.75pt,yscale=-1,xscale=1]

\draw [color={rgb, 255:red, 67; green, 104; blue, 152 }  ,draw opacity=1 ][line width=1.5]    (20,123) .. controls (59.2,93.6) and (406.29,108.62) .. (474.75,95.11) ;
\draw [shift={(478.58,94.25)}, rotate = 165.5] [fill={rgb, 255:red, 67; green, 104; blue, 152 }  ,fill opacity=1 ][line width=0.08]  [draw opacity=0] (13.4,-6.43) -- (0,0) -- (13.4,6.44) -- (8.9,0) -- cycle    ;
\draw [shift={(20,123)}, rotate = 323.13] [color={rgb, 255:red, 67; green, 104; blue, 152 }  ,draw opacity=1 ][fill={rgb, 255:red, 67; green, 104; blue, 152 }  ,fill opacity=1 ][line width=1.5]      (0, 0) circle [x radius= 4.36, y radius= 4.36]   ;
\draw [color={rgb, 255:red, 245; green, 166; blue, 35 }  ,draw opacity=1 ][line width=1.5]    (20,123) .. controls (68.13,65.7) and (185.58,117.25) .. (271.58,112.25) .. controls (357.58,107.25) and (322.42,93.36) .. (334.58,93.25) .. controls (346.75,93.14) and (368.09,114.48) .. (381.58,112.25) .. controls (394.54,110.11) and (444.37,148.02) .. (472.22,162.56) ;
\draw [shift={(475.58,164.25)}, rotate = 205.71] [fill={rgb, 255:red, 245; green, 166; blue, 35 }  ,fill opacity=1 ][line width=0.08]  [draw opacity=0] (13.4,-6.43) -- (0,0) -- (13.4,6.44) -- (8.9,0) -- cycle    ;
\draw  [draw opacity=0] (458.3,66.65) .. controls (476.39,74.53) and (489.58,99.07) .. (489.58,128.13) .. controls (489.58,158.88) and (474.8,184.58) .. (455.07,190.81) -- (445.79,128.13) -- cycle ; \draw  [color={rgb, 255:red, 0; green, 0; blue, 0 }  ,draw opacity=1 ] (458.3,66.65) .. controls (476.39,74.53) and (489.58,99.07) .. (489.58,128.13) .. controls (489.58,158.88) and (474.8,184.58) .. (455.07,190.81) ;  
\draw  [draw opacity=0][dash pattern={on 4.5pt off 4.5pt}] (352.3,64.65) .. controls (370.39,72.53) and (383.58,97.07) .. (383.58,126.13) .. controls (383.58,156.88) and (368.8,182.58) .. (349.07,188.81) -- (339.79,126.13) -- cycle ; \draw  [color={rgb, 255:red, 0; green, 0; blue, 0 }  ,draw opacity=1 ][dash pattern={on 4.5pt off 4.5pt}] (352.3,64.65) .. controls (370.39,72.53) and (383.58,97.07) .. (383.58,126.13) .. controls (383.58,156.88) and (368.8,182.58) .. (349.07,188.81) ;  
\draw  [dash pattern={on 0.84pt off 2.51pt}]  (20,123) -- (372.6,154.98) ;
\draw [shift={(375.58,155.25)}, rotate = 185.18] [fill={rgb, 255:red, 0; green, 0; blue, 0 }  ][line width=0.08]  [draw opacity=0] (8.93,-4.29) -- (0,0) -- (8.93,4.29) -- cycle    ;
\draw [shift={(20,123)}, rotate = 5.18] [color={rgb, 255:red, 0; green, 0; blue, 0 }  ][fill={rgb, 255:red, 0; green, 0; blue, 0 }  ][line width=0.75]      (0, 0) circle [x radius= 3.35, y radius= 3.35]   ;
\draw  [draw opacity=0][line width=1.5]  (483.57,95.68) .. controls (487.39,105.2) and (489.58,116.29) .. (489.58,128.13) .. controls (489.58,144.16) and (485.56,158.83) .. (478.92,170.07) -- (445.79,128.13) -- cycle ; \draw [color={rgb, 255:red, 189; green, 16; blue, 224 }  ,draw opacity=1 ][line width=1.5]    (483.57,95.68) .. controls (487.39,105.2) and (489.58,116.29) .. (489.58,128.13) .. controls (489.58,144.16) and (485.56,158.83) .. (478.92,170.07) ; \draw [shift={(478.92,170.07)}, rotate = 116.15] [color={rgb, 255:red, 189; green, 16; blue, 224 }  ,draw opacity=1 ][fill={rgb, 255:red, 189; green, 16; blue, 224 }  ,fill opacity=1 ][line width=1.5]      (0, 0) circle [x radius= 4.36, y radius= 4.36]   ; \draw [shift={(483.57,95.68)}, rotate = 72.61] [color={rgb, 255:red, 189; green, 16; blue, 224 }  ,draw opacity=1 ][fill={rgb, 255:red, 189; green, 16; blue, 224 }  ,fill opacity=1 ][line width=1.5]      (0, 0) circle [x radius= 4.36, y radius= 4.36]   ;
\draw  [color={rgb, 255:red, 139; green, 87; blue, 42 }  ,draw opacity=1 ][line width=1.5] [line join = round][line cap = round] (484.58,94.25) .. controls (499.64,94.25) and (491.48,100.15) .. (493.58,116.25) .. controls (494.37,122.29) and (506.45,127.73) .. (506.58,128.25) .. controls (507.83,133.25) and (500.73,143.65) .. (499.58,148.25) .. controls (498.26,153.56) and (504.96,152.91) .. (502.58,159.25) .. controls (501.14,163.09) and (496.41,163.31) .. (493.58,164.25) .. controls (491.56,164.92) and (492.8,168.37) .. (492.58,169.25) .. controls (491.53,173.47) and (483.93,170.25) .. (479.58,170.25) ;

\draw (478.92,170.07) node [anchor=north west][inner sep=0.75pt]  [color={rgb, 255:red, 245; green, 166; blue, 35 }  ,opacity=1 ]  {$\alpha $};
\draw (482,69) node [anchor=north west][inner sep=0.75pt]  [color={rgb, 255:red, 67; green, 104; blue, 152 }  ,opacity=1 ]  {$\beta $};
\draw (422,188) node [anchor=north west][inner sep=0.75pt]  [color={rgb, 255:red, 0; green, 0; blue, 0 }  ,opacity=1 ]  {$\partial _{\infty } X$};
\draw (149,142) node [anchor=north west][inner sep=0.75pt]    {$t:=( \alpha \mid \beta )_{o}$};
\draw (509,112) node [anchor=north west][inner sep=0.75pt]  [font=\small,color={rgb, 255:red, 189; green, 16; blue, 224 }  ,opacity=1 ]  {$d_{vis}( \alpha ,\ \beta ) \ \asymp \ e^{-t}$};
\draw (505,150) node [anchor=north west][inner sep=0.75pt]  [font=\footnotesize,color={rgb, 255:red, 139; green, 87; blue, 42 }  ,opacity=1 ]  {$d_{log}( \alpha ,\ \beta ) \ \asymp \ ( 1+t)^{-\theta}$};
\draw (12,126) node [anchor=north west][inner sep=0.75pt]    {$o$};

\end{tikzpicture}

    \caption{Denoting $t:=(\alpha\mid\beta)_o$, $d_{\mathrm{vis}}(\alpha, \beta)\asymp e^{-\varepsilon t}$ and $d_{\mathrm{log}}(\alpha, \beta)\asymp (1+t)^{-\theta}$}
    \label{fig:placeholder}
\end{figure}

\begin{corollary}
\label{cor:log-metric-gromov-product}
Let $\theta>0$ be an admissible snowflaking exponent. Then for all distinct
$\bfa,\bfb\in \bdy X$,
\[
d_{\mathrm{log}}(\bfa,\bfb)
\asymp
\frac{1}
{\bigl(1+(\bfa\mid\bfb)_\go\bigr)^\theta}.
\]
The comparison constants depend only on $\delta$, $\kappa$, $\theta$, and the
constants in the construction of the log metric.
\end{corollary}

\begin{proof}
By the definition of the logarithmic metric,
\[
d_{\mathrm{log}}(\bfa,\bfb)
\asymp
\frac{1}
{\bigl(1+r_\go(\bfa,\bfb)\bigr)^\theta}.
\]
By Corollary~\ref{cor:one-plus-r-one-plus-gromov-product},
$
1+r_\go(\bfa,\bfb)
\asymp
1+(\bfa\mid\bfb)_\go.
$
Combining the two comparisons proves the claim.
\end{proof}

The previous corollary is the main reason for the term ``logarithmic metric.'' Indeed, visual metrics decay exponentially in the Gromov product, while logarithmic metrics decay polynomially in the Gomov product as illustrated in Figure~\ref{fig:placeholder}. The log metric is obtained from a visual metric by a logarithmic change of scale.

We now identify the logarithmic metric with the polynomial Floyd metric
introduced in Section~\ref{subsubsec:polynomial-floyd}.

\begin{corollary}
\label{cor:log-metric-floyd}
Let $X$ be a proper geodesic $\delta$-hyperbolic space, and let
$\theta>0$ be admissible for the logarithmic metric. Then the natural
homeomorphism $J_\theta:(\partial X,d_{\mathrm{log}})\to(\partial_{F,\theta}X,\sigma_\theta)$ is bi-Lipschitz.
\end{corollary}
\begin{proof}
Fix distinct $\bfa,\bfb\in\partial X$, and let $\alpha_0,\beta_0$ be geodesic representatives based at $\go$. Let $P=(\bfa\mid\bfb)_\go$ and
\[
L
=
\liminf_{n,m\to\infty}
(\alpha_0(n)\mid\beta_0(m))_\go.
\]
By \eqref{eq:BH_upper_lower_bound_P}, we have
$P-2\delta\leq L\leq P$.

Let
$Q=\langle J_\theta(\bfa)\mid J_\theta(\bfb)
\rangle_{\go,F,\theta}$. Since the sequences
$(\alpha_0(n))$ and $(\beta_0(m))$ converge in the Floyd completion to
$J_\theta(\bfa)$ and $J_\theta(\bfb)$, respectively, the definition of
$Q$ gives $Q\leq L$. On the other hand, since $g_\theta$ is a
sphericalization function, the second assertion of
\cite[Theorem~2.11]{4BK09:floydideal} gives
$L\leq C(1+Q)$, where $C$ depends only on $\delta$ and $\theta$.
Therefore, $Q\leq P$ and $P\leq C(1+Q)+2\delta$, which imply $1+P\asymp1+Q$.

The Floyd comparison from
Section~\ref{subsubsec:polynomial-floyd} and
Corollary~\ref{cor:log-metric-gromov-product} now give
\[
\sigma_\theta\bigl(J_\theta(\bfa),J_\theta(\bfb)\bigr)
\asymp
(1+Q)^{-\theta}
\asymp
(1+P)^{-\theta}
\asymp
d_{\mathrm{log}}(\bfa,\bfb).
\]
Thus $J_\theta$ is bi-Lipschitz.
\end{proof}

\begin{corollary}
\label{cor:full-log-gromov-product}
Let $\kappa_0\equiv1$, and let $\theta>0$ be admissible. Under the
identification
$
\partial_{\mathrm{sl}}X
=
\partial_{\kappa_0}X
=
\partial X,
$
the metric $d^{X,\kappa_0}_{\go,\theta}$ induces the usual topology and
satisfies
\[
d^{X,\kappa_0}_{\go,\theta}(\bfa,\bfb)
\asymp
\frac{1}
{\bigl(1+(\bfa\mid\bfb)_\go\bigr)^\theta}.
\]
Moreover, for every sublinear function $\kappa$, the identity map
$
\id:
\bigl(\partial X,d^{X,\kappa_0}_{\go,\theta}\bigr)
\longrightarrow
\bigl(\partial X,d^{X,\kappa}_{\go,\theta}\bigr)
$
is bi-Lipschitz.
\end{corollary}
\begin{proof}
The first assertion follows from
Theorem~\ref{thm:log-metric-background}, applied with
$\kappa=\kappa_0$, together with
Corollary~\ref{cor:log-metric-gromov-product}.
For an arbitrary sublinear function $\kappa$, both
$d^{X,\kappa_0}_{\go,\theta}$ and
$d^{X,\kappa}_{\go,\theta}$ are comparable to
$
\bigl(1+(\bfa\mid\bfb)_\go\bigr)^{-\theta}.
$
Therefore, the identity map between the two metric spaces is
bi-Lipschitz.
\end{proof}

\begin{corollary}
\label{cor:change-geodesic-section}
Let $\sigma$ and $\tau$ be two choices of based geodesic representatives of the points of $\partial X$. Denote by $r_\go^\sigma$ and $r_\go^\tau$ the corresponding symmetric $\kappa$-fellow-traveling radii. Then there exists $B=B(\delta,\kappa)\ge 1$ such that, for all distinct $\bfa,\bfb\in\partial X$,
\[
B^{-1}\bigl(1+r_\go^\tau(\bfa,\bfb)\bigr)
\le 
1+r_\go^\sigma(\bfa,\bfb)
\le 
B\bigl(1+r_\go^\tau(\bfa,\bfb)\bigr).
\]
Consequently, if
$d^\sigma_{\mathrm{log}}$ and $d^\tau_{\mathrm{log}}$ are the corresponding logarithmic metrics, constructed using the same logarithmic base $\lbase$ and the same admissible exponent $\theta$, then
$
\id:
(\bdy X,d^\sigma_{\mathrm{log}})
\longrightarrow
(\bdy X,d^\tau_{\mathrm{log}})
$
is bi-Lipschitz.
\end{corollary}

\begin{proof}
Let $P=(\bfa\mid\bfb)_\go$. The proof of Proposition~\ref{prop:relationship_p_r} applies to every choice of based geodesic representatives. Moreover, its constant is independent of that choice, since every geodesic center is equipped with the common gauge $\sM$. Hence, there exists $C=C(\delta,\kappa)\ge 1$ such that
\begin{align*}
P
&\le 
r_\go^\sigma(\bfa,\bfb)
\le 
P+C\kappa(P)+C, \text{ and }\\
P
&\le 
r_\go^\tau(\bfa,\bfb)
\le 
P+C\kappa(P)+C.
\end{align*}

Moreover, since $\kappa$ is sublinear,
$
\sup_{t\ge 0}\frac{\kappa(t)}{1+t}
<
\infty.
$
Denoting $L_\kappa = \sup_{t\ge 0}\frac{\kappa(t)}{1+t}$, we obtain,
\[
\begin{aligned}
1+r_\go^\sigma(\bfa,\bfb)
&\le 
1+P+C\kappa(P)+C  \\
&\le 
\bigl(1+CL_\kappa+C\bigr)(1+P)  \\
&\le 
\bigl(1+CL_\kappa+C\bigr)
\bigl(1+r_\go^\tau(\bfa,\bfb)\bigr).
\end{aligned}
\]
Interchanging $\sigma$ and $\tau$ proves the first claim, with $B=1+CL_\kappa+C$.

Next, we choose constants
$0<c_{\log}\le  C_{\log}<\infty$ such that, for
$\nu\in\{\sigma,\tau\}$,
\[
c_{\log}
\bigl(1+r_\go^\nu(\bfa,\bfb)\bigr)^{-\theta}
\le 
d^\nu_{\mathrm{log}}(\bfa,\bfb)
\le 
C_{\log}
\bigl(1+r_\go^\nu(\bfa,\bfb)\bigr)^{-\theta}.
\]
The comparison of the two radii gives
\[
\frac{c_{\log}}{C_{\log}}B^{-\theta}
d^\tau_{\mathrm{log}}(\bfa,\bfb)
\le 
d^\sigma_{\mathrm{log}}(\bfa,\bfb)
\le 
\frac{C_{\log}}{c_{\log}}B^\theta
d^\tau_{\mathrm{log}}(\bfa,\bfb).
\]
\end{proof}

Thus the bi-Lipschitz class of $d_{\log}$ is independent of the choice
of geodesic representatives. Accordingly, we suppress this choice from
the notation throughout the remainder of the paper.

\begin{corollary}
\label{cor:visual-log-comparison}
Let $d_{\mathrm{vis}}$ be a visual metric on $\bdy X$ with parameter
$\varepsilon>0$, and let $d_{\mathrm{log}}$ be a log metric with admissible exponent $\theta>0$. Then there exists a constant $B>1$, such that for all
distinct $\bfa,\bfb\in \bdy X$,
\[
d_{\mathrm{log}}(\bfa,\bfb)
\asymp
\frac{1}
{\bigl(1+\log(B/d_{\mathrm{vis}}(\bfa,\bfb))\bigr)^\theta}.
\]
\end{corollary}

\begin{proof}
Let $P=(\bfa\mid\bfb)_\go$. Since $d_{\mathrm{vis}}$ is a visual metric with parameter $\varepsilon$, there
exist constants $c_{\mathrm{vis}},C_{\mathrm{vis}}>0$ such that
\[
c_{\mathrm{vis}}e^{-\varepsilon P}
\le 
d_{\mathrm{vis}}(\bfa,\bfb)
\le 
C_{\mathrm{vis}}e^{-\varepsilon P}.
\]
Choose $B>C_{\mathrm{vis}}$. Then
\[
\log\left(\frac{B}{C_{\mathrm{vis}}}\right)+\varepsilon P
\le 
\log\left(\frac{B}{d_{\mathrm{vis}}(\bfa,\bfb)}\right)
\le 
\log\left(\frac{B}{c_{\mathrm{vis}}}\right)+\varepsilon P.
\]
Therefore,
\[
1+\log\left(\frac{B}{d_{\mathrm{vis}}(\bfa,\bfb)}\right)
\asymp
1+P.
\]
By Corollary~\ref{cor:log-metric-gromov-product},
\[
d_{\mathrm{log}}(\bfa,\bfb)
\asymp
\frac{1}{(1+P)^\theta}.
\]
Combining the two estimates yields the result.
\end{proof}

\section{Quasi-isometry implies bi-Lipschitz equivalence}
\label{S:qi-invariance}

Visual metrics on the boundary of a hyperbolic space are usually preserved by quasi-isometries only up to quasisymmetry or quasi-M\"obius equivalence. The logarithmic metric has stronger relationship: once the same snowflaking exponent is chosen on both boundaries, quasi-isometries induce bi-Lipschitz boundary maps. Moreover, the choice of snowflaking exponent does not depend on any property of the underlying space. In fact, it is dependent only on the choice of base $\lbase$.

Throughout this section, $\theta>0$ denotes an admissible exponent and $\kappa$ is a fixed sublinear function. For a hyperbolic space $(X,\go)$, we write $d^X_{\go,\theta}$ for its logarithmic metric following conventions established in \S~\ref{S:background}.

\begin{theorem}
\label{thm:qi_implies_bl}
Let $X$ and $Y$ be proper geodesic hyperbolic spaces, and let $\Phi: X\to Y$ be a $(q,Q)$-quasi-isometric embedding. Fix $\go\in X$ and let $\go'=\Phi(\go)$. 

Then the induced boundary map
\[
\Phi_*\colon
(\bdy X,d^X_{\go,\theta})
\longrightarrow
(\bdy Y,d^Y_{\go',\theta})
\]
is a bi-Lipschitz embedding. In particular, it is bi-Lipschitz onto its
image. If $\Phi$ is a quasi-isometry, then $\Phi_*$ is surjective and
hence is a bi-Lipschitz equivalence.
\end{theorem}

\textit{Proof.} For $\bfa, \bfb \in \bdy X$, let $p=(\bfa\mid\bfb)_{\go}$, and $p'=
\bigl(\Phi_*(\bfa)\mid\Phi_*(\bfb)\bigr)_{\go'}$. By Lemma~\ref{lem:QI-Gromov-product}, there exist $A\ge 1$ and $B\ge 0$ such that
\[
\frac1A p-B\le  p'\le  Ap+B. 
\]
Consequently, for some $M\ge 1$,
\[
M^{-1}(1+p)\le 1+p'\le  M(1+p). 
\]
Indeed, one may take $
M=\max\{A+B+1,A+AB+1\}$. By Corollary~\ref{cor:log-metric-gromov-product},
\begin{align*}
d^X_{\go,\theta}(\bfa,\bfb)
&\asymp(1+p)^{-\theta}, \quad {\text{and }}\\
d^Y_{\go',\theta}
\bigl(\Phi_*(\bfa),\Phi_*(\bfb)\bigr)
&\asymp(1+p')^{-\theta}.
\end{align*}

The comparison between $1+p$ and $1+p'$ therefore gives a constant $L\ge 1$ such that
\[
L^{-1}d^X_{\go,\theta}(\bfa,\bfb)
\le 
d^Y_{\go',\theta}
\bigl(\Phi_*(\bfa),\Phi_*(\bfb)\bigr)
\le 
L\,d^X_{\go,\theta}(\bfa,\bfb).
\]
Since $\Phi_*$ is injective, it is a bi-Lipschitz embedding.

If $\Phi$ is a quasi-isometry, its induced boundary map is surjective,
so it is a bi-Lipschitz equivalence. \hfill $\qed$

\begin{figure}
\centering
    
    \tikzset{every picture/.style={line width=0.75pt}} 
    
        \begin{tikzpicture}[x=0.75pt,y=0.75pt,yscale=-1,xscale=1]
        
        \draw [color={rgb, 255:red, 184; green, 233; blue, 134 }  ,draw opacity=1 ][fill={rgb, 255:red, 184; green, 233; blue, 134 }  ,fill opacity=1 ][line width=1.5] [line join = round][line cap = round]   (226.1,6411.53) .. controls (219.88,6411.53) and (212.09,6411.3) .. (208.1,6417.53) .. controls (204.69,6422.87) and (208.67,6430.23) .. (208.1,6436.53) .. controls (207.83,6439.49) and (194.51,6448.44) .. (193.1,6453.53) .. controls (189.1,6467.92) and (201.18,6483.55) .. (215.1,6486.53) .. controls (229.53,6489.63) and (226.7,6468.29) .. (238.1,6467.53) .. controls (247.41,6466.91) and (256.86,6468.85) .. (266.1,6467.53) .. controls (273.97,6466.41) and (280,6451.03) .. (281.1,6445.53) .. controls (282.02,6440.96) and (280.29,6436.13) .. (281.1,6431.53) .. controls (281.35,6430.14) and (283.71,6429.89) .. (284.1,6428.53) .. controls (284.72,6426.37) and (282.01,6418.13) .. (281.1,6416.53) .. controls (273.84,6403.82) and (245.11,6403.53) .. (236.1,6412.53) .. controls (233.61,6415.03) and (230.43,6406.86) .. (230.1,6406.53) .. controls (229.42,6405.85) and (224.29,6410.34) .. (223.1,6411.53) ;
        \draw [color={rgb, 255:red, 67; green, 104; blue, 152 }  ,draw opacity=1 ][fill={rgb, 255:red, 67; green, 104; blue, 152 }  ,fill opacity=1 ][line width=1.5] [line join = round][line cap = round]   (448.1,6412.53) .. controls (429.08,6412.53) and (407.66,6415.14) .. (404.1,6436.53) .. controls (403.13,6442.38) and (401,6444.76) .. (402.1,6450.53) .. controls (403.07,6455.62) and (404.38,6460.65) .. (406.1,6465.53) .. controls (408.64,6472.73) and (416.72,6474.17) .. (422.1,6477.53) .. controls (423.55,6478.44) and (426.36,6484.58) .. (428.1,6483.53) .. controls (432.73,6480.75) and (444.32,6457.19) .. (445.1,6452.53) .. controls (445.48,6450.27) and (444.29,6448.34) .. (446.1,6446.53) .. controls (453.68,6438.95) and (465.68,6449.46) .. (468.1,6432.53) .. controls (469.11,6425.48) and (466.12,6419.34) .. (459.1,6416.53) .. controls (457.34,6415.83) and (449.1,6413.47) .. (449.1,6412.53) ;
        \draw    (284,6437.23) .. controls (323.6,6407.53) and (357.52,6416.65) .. (398.85,6438.86) ;
        \draw [shift={(400.1,6439.53)}, rotate = 208.45] [color={rgb, 255:red, 0; green, 0; blue, 0 }  ][line width=0.75]    (10.93,-3.29) .. controls (6.95,-1.4) and (3.31,-0.3) .. (0,0) .. controls (3.31,0.3) and (6.95,1.4) .. (10.93,3.29)   ;
        \draw  [dash pattern={on 0.84pt off 2.51pt}]  (242.1,6468.53) -- (242.1,6545.53) ;
        \draw [shift={(242.1,6548.53)}, rotate = 270] [fill={rgb, 255:red, 0; green, 0; blue, 0 }  ][line width=0.08]  [draw opacity=0] (8.93,-4.29) -- (0,0) -- (8.93,4.29) -- cycle    ;
        \draw  [dash pattern={on 0.84pt off 2.51pt}]  (445.1,6449.53) -- (445.1,6543.53) ;
        \draw [shift={(445.1,6546.53)}, rotate = 270] [fill={rgb, 255:red, 0; green, 0; blue, 0 }  ][line width=0.08]  [draw opacity=0] (8.93,-4.29) -- (0,0) -- (8.93,4.29) -- cycle    ;
        \draw  [color={rgb, 255:red, 126; green, 211; blue, 33 }  ,draw opacity=1 ][line width=1.5] [line join = round][line cap = round] (233.1,6550.53) .. controls (226.88,6550.53) and (219.09,6550.3) .. (215.1,6556.53) .. controls (211.69,6561.87) and (215.67,6569.23) .. (215.1,6575.53) .. controls (214.83,6578.49) and (201.51,6587.44) .. (200.1,6592.53) .. controls (196.1,6606.92) and (208.18,6622.55) .. (222.1,6625.53) .. controls (236.53,6628.63) and (233.7,6607.29) .. (245.1,6606.53) .. controls (254.41,6605.91) and (263.86,6607.85) .. (273.1,6606.53) .. controls (280.97,6605.41) and (287,6590.03) .. (288.1,6584.53) .. controls (289.02,6579.96) and (287.29,6575.13) .. (288.1,6570.53) .. controls (288.35,6569.14) and (290.71,6568.89) .. (291.1,6567.53) .. controls (291.72,6565.37) and (289.01,6557.13) .. (288.1,6555.53) .. controls (280.84,6542.82) and (252.11,6542.53) .. (243.1,6551.53) .. controls (240.61,6554.03) and (237.43,6545.86) .. (237.1,6545.53) .. controls (236.42,6544.85) and (231.29,6549.34) .. (230.1,6550.53) ;
        \draw  [color={rgb, 255:red, 67; green, 104; blue, 152 }  ,draw opacity=0.79 ][line width=1.5] [line join = round][line cap = round] (461.1,6545.53) .. controls (442.08,6545.53) and (420.66,6548.14) .. (417.1,6569.53) .. controls (416.13,6575.38) and (414,6577.76) .. (415.1,6583.53) .. controls (416.07,6588.62) and (417.38,6593.65) .. (419.1,6598.53) .. controls (421.64,6605.73) and (429.72,6607.17) .. (435.1,6610.53) .. controls (436.55,6611.44) and (439.36,6617.58) .. (441.1,6616.53) .. controls (445.73,6613.75) and (457.32,6590.19) .. (458.1,6585.53) .. controls (458.48,6583.27) and (457.29,6581.34) .. (459.1,6579.53) .. controls (466.68,6571.95) and (478.68,6582.46) .. (481.1,6565.53) .. controls (482.11,6558.48) and (479.12,6552.34) .. (472.1,6549.53) .. controls (470.34,6548.83) and (462.1,6546.47) .. (462.1,6545.53) ;
        \draw    (290.1,6580.53) -- (412,6579.25) ;
        \draw [shift={(414,6579.23)}, rotate = 179.4] [color={rgb, 255:red, 0; green, 0; blue, 0 }  ][line width=0.75]    (10.93,-3.29) .. controls (6.95,-1.4) and (3.31,-0.3) .. (0,0) .. controls (3.31,0.3) and (6.95,1.4) .. (10.93,3.29)   ;
        \draw    (282.1,6659.53) -- (404,6658.25) ;
        \draw [shift={(406,6658.23)}, rotate = 179.4] [color={rgb, 255:red, 0; green, 0; blue, 0 }  ][line width=0.75]    (10.93,-3.29) .. controls (6.95,-1.4) and (3.31,-0.3) .. (0,0) .. controls (3.31,0.3) and (6.95,1.4) .. (10.93,3.29)   ;
        
        \draw (229,6379.23) node [anchor=north west][inner sep=0.75pt]    {$X$};
        \draw (412,6386.23) node [anchor=north west][inner sep=0.75pt]    {$Y$};
        \draw (295,6394.23) node [anchor=north west][inner sep=0.75pt]  [font=\small] [align=left] {quasi-isometry};
        \draw (325,6418.23) node [anchor=north west][inner sep=0.75pt]    {$f$};
        \draw (217,6568.23) node [anchor=north west][inner sep=0.75pt]    {$( \partial X,\ d_{vis})$};
        \draw (427,6556.23) node [anchor=north west][inner sep=0.75pt]    {$( \partial Y,\ d_{vis})$};
        \draw (336,6545.23) node [anchor=north west][inner sep=0.75pt]    {$f_{*}^{vis}$};
        \draw (297.1,6580.53) node [anchor=north west][inner sep=0.75pt]  [font=\small] [align=left] {quasisymmetric};
        \draw (203,6644.23) node [anchor=north west][inner sep=0.75pt]    {$( \partial X,\ \textcolor[rgb]{0.88,0.06,0.39}{d}\textcolor[rgb]{0.88,0.06,0.39}{_{log}})$};
        \draw (337,6629.23) node [anchor=north west][inner sep=0.75pt]  [color={rgb, 255:red, 224; green, 16; blue, 142 }  ,opacity=1 ]  {$f_{*}^{log}$};
        \draw (418,6645.23) node [anchor=north west][inner sep=0.75pt]    {$( \partial Y,\ \textcolor[rgb]{0.88,0.06,0.47}{d}\textcolor[rgb]{0.88,0.06,0.47}{_{log}})$};
        \draw (306.1,6660.73) node [anchor=north west][inner sep=0.75pt]  [font=\small,color={rgb, 255:red, 224; green, 16; blue, 76 }  ,opacity=1 ] [align=left] {bi-Lipschitz};

        \end{tikzpicture}
\caption{$X,Y$ are Gromov hyperbolic spaces, then their boundaries with visual scaling are quasisymmetric and their boundaries with logarithmic/polynomial scaling are bi-Lispchitz equivalent.}
\end{figure}

By Theorem~\ref{thm:qi_implies_bl}, if $G$ and $H$ are quasi-isometric hyperbolic groups, then their log boundaries are bi-Lipschitz equivalent. In particular, changing the finite generating set of a hyperbolic group changes the associated log metric only up to bi-Lipschitz equivalence.

Next, we record how the change of auxiliary constants changes the log metric. Since the change of basepoint is a quasi-isometry, therefore, we immediately obtain the following Corollary.

\begin{corollary}
\label{prop:change-basepoint-log}
Let $X$ be a proper geodesic $\delta$-hyperbolic space, and let
$\go,\go'\in X$. Let $\theta>0$ be an admissible snowflaking exponent. Then the
identity map
$
\id:
(\bdy X,d^{X}_{\go,\theta})
\longrightarrow
(\bdy X,d^{X}_{\go',\theta})
$ is bi-Lipschitz.
\end{corollary}

\begin{corollary}
\label{cor:qi-different-exponents-qs}
Let $X$, $Y$, $\Phi$, $\go$, and $\go'$ be as in
Theorem~\ref{thm:qi_implies_bl}. Let $\theta,\theta'>0$ be admissible snowflaking exponents. Then $\Phi_*: (\bdy X,d^X_{\go,\theta})
\longrightarrow (\bdy Y,d^Y_{\go',\theta'})$ is quasisymmetric.
\end{corollary}

\begin{proof}
By Theorem~\ref{thm:qi_implies_bl} and Corollary~\ref{prop:change-basepoint-log}, $\Phi_*:
(\bdy X,d^X_{\go,\theta})
\longrightarrow
(\bdy Y,d^Y_{\go',\theta})$ is bi-Lipschitz. By \cite[Proposition~3.17]{gjq}, the identity map
$
    \id:
    (\bdy Y,d^Y_{\go',\theta})
    \longrightarrow
    (\bdy Y,d^Y_{\go',\theta'})$
is quasisymmetric with power-type distortion. Therefore, the composition
\[
(\bdy X,d^X_{\go,\theta})
\xrightarrow{\ \Phi_*\ }
(\bdy Y,d^Y_{\go',\theta})
\xrightarrow{\ \id\ }
(\bdy Y,d^Y_{\go',\theta'})
\]
is quasisymmetric.
\end{proof}

Moreover, we emphasize that the induced bi-Lipschitz map on the boundaries equipped with log metric is not merely an isometry or a similarity, that is, the map is a genuine bi-Lipschitz map. We illustrate this fact by the example below.
\begin{example}
    Let $\FF_2 = \langle a,b\rangle$ be a free group on two generators $a,b$, equipped with the word metric $d$ associated to $\{a^{\pm1},b^{\pm1}\}$. Consider the Neilsen automorphism $$\Phi: (\FF_2, d) \longrightarrow (\FF_2, d), \quad  \Phi(a) = ab, \quad \Phi(b) = b.$$
    Its inverse is given by $\Phi^{-1}(a)=ab^{-1}$ and $\Phi^{-1}(b)=b$. 
    
    Since the images of the generators under both $\Phi$ and $\Phi^{-1}$ have word length at most $2$, we have
    \[
        \frac{1}{2}d(g,h)
        \le 
        d\bigl(\Phi(g),\Phi(h)\bigr)
        \le 
        2d(g,h)
    \]
    for all $g,h\in \FF_2$. Thus $\Phi$ is a quasi-isometry. It is not an
    isometry of the standard Cayley graph, since $d(e,a)=1$ but $d\bigl(e,\Phi(a)\bigr)=d(e,ab)=2$.

    Let $\bdy \FF_2$ denote the visual boundary comprised of infinite freely reduced words. For distinct $\xi,\eta\in\partial \FF_2$, let 
    \[
    N(\xi,\eta)
    =
    \max\{m\ge  0:\xi_1\cdots\xi_m=\eta_1\cdots\eta_m\}
    \]
    denote their common-prefix length. Then $N(\xi,\eta)=(\xi\mid\eta)_e$. Consequently, fixing $\theta >0$, we obtain  
    \[
    d_{\mathrm{log}}(\xi,\eta)
    =
    \begin{cases}
    0,&\xi=\eta,\\[1mm]
    \bigl(1+N(\xi,\eta)\bigr)^{-\theta},&\xi\neq\eta,
    \end{cases}
\]

    Let $\Phi_*: (\bdy \FF_2, d_{\mathrm{log}}) \longrightarrow (\bdy \FF_2, d_{\mathrm{log}})$ be the induced boundary bi-Lipschitz homeomorphism from Theorem~\ref{thm:qi_implies_bl}. Clearly, $\Phi_*$ is not an identity since $\Phi_*(a^\infty)=(ab)^\infty\neq a^\infty$.

    We claim that $\Phi_*$ is not an isometry.
    To see this, let $\xi = a^{\infty}$ and $\eta_n = a^n b^{\infty}$. Then $N(\xi, \eta_n) = n$. On the other hand, $\Phi_*(\xi) = (ab)^{\infty}$ and $\Phi_*(\eta_n) = (ab)^n b^{\infty}$. Then $N(\Phi_*(\xi), \Phi_*(\eta_n))= 2n$. Therefore,
    \[
        \frac{d_{\mathrm{log}}(\Phi_*(\xi), \Phi_*(\eta_n))}{d_{\mathrm{log}}(\xi, \eta_n)} = \left(\frac{n+1}{2n+1}\right)^{\theta} \longrightarrow 2^{-\theta} \quad \text{ as } n\to \infty
    \]
    Thus $\Phi_*$ shrinks arbitrarily small distances by at least a factor $2^{-\theta}$. 
    
    The same map also stretches arbitrarily small distances. 
    Let $\xi = (ab^{-1})^{\infty}$ and $\eta_n =(ab^{-1})^n (b^{-1})^\infty$
    Then $N(\xi, \eta_n)= 2n$. Since $\Phi(ab^{-1}) = abb^{-1} = a$, we obtain $\Phi_*(\xi) =a^{\infty}$ and 
    $\Phi_*(\eta_n) =a^{n}(b^{-1})^{\infty}$. It follows that $N(\Phi_*(\xi), \Phi_*(\eta_n))= n$, and hence,
    \[
        \frac{d_{\mathrm{log}}(\Phi_*(\xi), \Phi_*(\eta_n))}{d_{\mathrm{log}}(\xi, \eta_n)} = \left(\frac{2n+1}{n+1}\right)^{\theta} \longrightarrow 2^{\theta} \quad \text{ as } n\to \infty
    \]

Thus $\Phi_*$ both stretches and shrinks distances. In particular, it is neither an isometry nor a similarity, and every bi-Lipschitz constant for $\Phi_*$ is at least $2^\theta$.
\end{example}

\subsection{Quasiconvexity}
Geometry of subgroups of a $\delta$-hyperbolic group gives a lot of coarse information about the group itself. One such class of subgroups is that of quasiconvex subgroups. As an application, we obtain a logarithmic-metric characterization of quasiconvex subgroups.

Recall that if $G$ is a word-hyperbolic group with finite generating set $S$,
a subgroup $H\le  G$ is called \textit{{quasiconvex}} if there exists $R\ge  0$
such that every geodesic in the Cayley graph $\Gamma(G,S)$ of $G$ (with respect to the generating set $S$) with endpoints in
$H$ is contained in the $R$-neighborhood of $H$. 

This definition is independent
of the choice of finite generating set. If $H$ is finitely generated, then
$H$ is quasiconvex in $G$ if and only if it is quasi-isometrically embedded; see \cite[III.$\Gamma$.3]{book:BridsonHaefliger}.

We can also phrase quasiconvexity in terms of Cannon-Thurston maps. Suppose now
that $H\le  G$ is a finitely generated word-hyperbolic subgroup and that the
inclusion $i:H\hookrightarrow G$ admits a Cannon-Thurston map, that is, a
continuous extension
$
\widehat{i}:H\cup\partial H\longrightarrow G\cup\partial G
$ whose restriction to $H$ is the inclusion and whose boundary restriction is
$
\partial i:\partial H\longrightarrow \partial G.
$ Cannon-Thurston maps do not exist for all hyperbolic subgroups of hyperbolic groups; for instance, see \cite{noCTmap:bakerriley}.

\begin{corollary}
\label{cor:characterization_quasiconvexity}
Let $G$ be a word-hyperbolic group, and let $H\le  G$ be a finitely generated
word-hyperbolic subgroup. Suppose that the inclusion $i:H\hookrightarrow G$
admits a Cannon-Thurston map
$\partial i:\partial H\longrightarrow \partial G.$
Let $\theta>0$ be a common admissible exponent for the log metrics on
$\partial H$ and $\partial G$. Then the following are equivalent:

\begin{enumerate}
    \item $H$ is quasiconvex in $G$.
    \item The Cannon-Thurston map
    $
    \partial i:
    (\partial H,d^H_{e,\theta})
    \longrightarrow
    (\partial G,d^G_{e,\theta})
    $
    is bi-Lipschitz onto its image.
    \item The Cannon-Thurston map $\partial i:\partial H\to\partial G$ is
    injective.
\end{enumerate}
\end{corollary}

\begin{proof}
If $H$ is finite, then $\partial H=\varnothing$ and all three assertions
hold trivially. If $H$ is infinite and virtually cyclic, then $H$ is
quasiconvex in $G$ and $\partial H$ consists of two points. The
Cannon-Thurston map sends these points to the distinct attracting and
repelling fixed points of an infinite-order element of $H$. Hence it is
injective and bi-Lipschitz onto its image.

We may therefore assume that $H$ is non-elementary. If $H$ is quasiconvex in $G$, then
the inclusion $i:H\hookrightarrow G$ is a quasi-isometric embedding and hence, $\partial i:
    (\partial H,d^H_{e, \theta})
    \longrightarrow
    (\partial G,d^G_{e,\theta})$ induced by $i$ is a bi-Lipschitz map onto its image by Theorem~\ref{thm:qi_implies_bl}. Since the Cannon-Thurston map is the continuous boundary extension of the same inclusion, this boundary map is precisely $\partial i$. Hence $(1)$ implies $(2)$. The implication $(2)\Rightarrow(3)$ is immediate, since a bi-Lipschitz map onto its image is injective.
    Finally, assume that $\partial i$ is injective. By \cite[Proposition~2.13]{conicallimit:jeonkapovichleiningeroshika}, equivalently by \cite[Lemma~2.5]{Mahan:commensurators}, $H$ is quasiconvex in $G$. Hence $(3)$ implies $(1)$.
\end{proof}

\section{Metric-analysis preliminaries}
\label{S:metric-analysis-preliminaries}

In this section we collect the metric-analysis tools used in Section~\ref{S:metric-consequences}.

\subsection{Hausdorff gauge measures}
\label{subsec:hausdorff-gauge-measures}

We first recall Hausdorff measures associated to general gauge functions. These gauge functions are different from the Morse gauge used in the definition of sublinear Morse boundary or the conformal gauge of metric. 

A \textit{gauge function} is a non-decreasing continuous function $\varphi:[0,\infty)\to[0,\infty)$ such that $\varphi(0)=0$.

Let $(Z,d)$ be a metric space and let $A\subseteq Z$. For $\delta>0$, define
\[
\mathcal H^\varphi_\delta(A)
=
\inf
\left\{
\sum_{i=1}^\infty \varphi(\operatorname{diam} A_i):
A\subseteq \bigcup_{i=1}^\infty A_i,\ 
\operatorname{diam} A_i\le  \delta
\right\}.
\]
The corresponding Hausdorff gauge measure is
\[
\mathcal H^\varphi(A)
=
\lim_{\delta\downarrow 0}\mathcal H^\varphi_\delta(A).
\]
The limit exists because $\mathcal H^\varphi_\delta(A)$ is non-decreasing as
$\delta\downarrow 0$. When $\varphi_s(t)=t^s$, we recover the usual
$s$-dimensional Hausdorff measure:
\[
\mathcal H^s=\mathcal H^{\varphi_s}.
\]

If $\varphi$ and $\psi$ are gauge functions, we write
$\psi\prec\varphi$ if
\[
\lim_{r\downarrow 0}\frac{\psi(r)}{\varphi(r)}=0.
\]
Thus $\psi\prec\varphi$ means that $\psi$ decays faster than $\varphi$ near
$0$. For example, for power gauges $\varphi_s(t)=t^s$ and
$\varphi_\sigma(t)=t^\sigma$, one has
$\varphi_s\prec \varphi_\sigma
\,\Longleftrightarrow\,
s>\sigma$.

The comparison below is standard; see, for example,
\cite[\S 2.5]{book:falconer_fractalgeom}. We include the short proof for
convenience.
\begin{lemma}
\label{lem:compare-hausdorff-gauge-measure}
Let $\varphi$ and $\psi$ be gauge functions with $\psi\prec\varphi$. Then for
every subset $A$ of every metric space, the following hold:
\begin{enumerate}
    \item If $\mathcal H^\varphi(A)<\infty$, then $\mathcal H^\psi(A)=0$.
    \item If $\mathcal H^\psi(A)>0$, then $\mathcal H^\varphi(A)=\infty$.
\end{enumerate}
\end{lemma}

\begin{proof}
Since $\psi\prec\varphi$, for every $\varepsilon>0$ there exists
$\delta_0>0$ such that
\[
\psi(t)\le  \varepsilon\varphi(t)
\qquad
\text{for all }0<t\le  \delta_0.
\]
Let $0<\delta\le \delta_0$. For every cover
$A\subseteq\bigcup_i A_i$ with $\operatorname{diam} A_i\le \delta$, we have
\[
\sum_i \psi(\operatorname{diam} A_i)
\le 
\varepsilon
\sum_i \varphi(\operatorname{diam} A_i).
\]
Taking infima over all such covers gives $\mathcal H^\psi_\delta(A)
\le 
\varepsilon \mathcal H^\varphi_\delta(A)$. Letting $\delta\downarrow0$ gives $\mathcal H^\psi(A)
\le 
\varepsilon \mathcal H^\varphi(A)$. If $\mathcal H^\varphi(A)<\infty$, then letting $\varepsilon\downarrow0$ gives
$\mathcal H^\psi(A)=0$. The second statement is the contrapositive of the first.
\end{proof}

We will use the following special gauges. For $c_1,c_2,q>0$, define
\[
\psi_q(t)=
\begin{cases}
c_1\exp(-c_2t^{-q}), & t>0,\\
0, & t=0.
\end{cases}
\]
Then $\psi_q$ decays faster than every power gauge. Indeed, for every $s>0$,
\[
\lim_{t\downarrow0}
\frac{\psi_q(t)}{t^s}
=
\lim_{t\downarrow0}
c_1t^{-s}\exp(-c_2t^{-q})
=
0.
\]
Hence $\psi_q\prec \varphi_s$ for every $s>0$. Combining this with \ref{lem:compare-hausdorff-gauge-measure}(i), we conclude the following.
\begin{corollary}
\label{cor:positive-exponential-gauge-infinite-dim}
Let $(Z,d)$ be a metric space. If $\mathcal H^{\psi_q}(Z)>0$ for some gauge
function $\psi_q(t)=c_1\exp(-c_2t^{-q})$, then $\mathcal H^s(Z)=\infty$ for every $s>0$. In particular, $\dim_H(Z,d)=\infty$.
\end{corollary}

\subsection{Modulus of continuity and Hausdorff dimension}
\label{subsec:modulus-hausdorff-dimension}

Let $g:(Z,d_Z)\to(W,d_W)$ be a uniformly continuous map. We define the modulus of
continuity of $g$ as
\[
\omega_g(t)
:=
\sup\{d_W(g(a),g(b)):a,b\in Z,\ d_Z(a,b)\le  t\}.
\]
Then $\omega_g$ is non-decreasing, $\omega_g(0)=0$, and uniform continuity
implies $\omega_g(t)\to0$ as $t\downarrow0$. Moreover, for every subset
$E\subseteq Z$,
\[
\operatorname{diam}_W g(E)
\le 
\omega_g(\operatorname{diam}_Z E).
\]

The next lemma shows how Hausdorff measure changes under maps with controlled modulus of continuity. It is the general gauge-function version of the usual H\"older estimate; see
\cite[Proposition~2.2]{book:falconer_fractalgeom}, \cite[Chapter~2]{book:hdim:rogers}.
\begin{lemma}
\label{lem:measure-bound-modulus}
Let $g:(Z,d_Z)\to(W,d_W)$ be uniformly continuous, and let $\omega$ be a
non-decreasing function with $\omega(t)\to0$ as $t\downarrow0$ such that $d_W(g(a),g(b))\le  \omega(d_Z(a,b))$ for all $a,b\in Z$. Then for every gauge function $\varphi$ and every
$A\subseteq Z$,
\[
    \mathcal H^\varphi(g(A))
    \le 
    \mathcal H^{\varphi\circ\omega}(A).
\]
\end{lemma}

\begin{proof}
Let $A\subseteq Z$ and fix $\delta>0$. If
$A\subseteq\bigcup_i A_i$ with $\operatorname{diam}_Z A_i\le \delta$, then
$g(A)\subseteq\bigcup_i g(A_i)$ and
\[
\operatorname{diam}_W g(A_i)
\le 
\omega(\operatorname{diam}_Z A_i).
\]
Therefore
\[
\sum_i \varphi(\operatorname{diam}_W g(A_i))
\le 
\sum_i
\varphi\bigl(\omega(\operatorname{diam}_Z A_i)\bigr).
\]
Taking the infimum over all such covers and then letting $\delta\downarrow0$ gives
the desired inequality.
\end{proof}

\begin{corollary}
\label{cor:modulus-implies-infinite-Hausdorff-dim}
Let $g:(Z,d_Z)\to(W,d_W)$ be a surjective uniformly continuous map. Suppose that
there exist constants $c_1,c_2,q>0$ such that, for all sufficiently small
$t>0$, $\omega_g(t)\le  c_1\exp(-c_2t^{-q})$.
If $\dim_H(W,d_W)>0$, then $\dim_H(Z,d_Z)=\infty$.
\end{corollary}

\begin{proof}
Choose $0<\sigma<\dim_H(W,d_W)$. Then $\mathcal H^\sigma(W,d_W)>0$. Since $g$ is surjective, $g(Z)=W$. By Lemma~\ref{lem:measure-bound-modulus},
\[
0 < \, \mathcal H^\sigma(W,d_W)
=
\mathcal H^\sigma(g(Z))
\le 
\mathcal H^{\sigma\ \circ\ \omega_g}(Z,d_Z).
\]
Now $\sigma \circ \omega_g$ is also a gauge of the form $t\mapsto C_1\exp(-C_2t^{-q})$ for different choices of $C_1$ and $C_2$. By
Corollary~\ref{cor:positive-exponential-gauge-infinite-dim}, $\dim_H(Z,d_Z)=\infty$.

\end{proof}

\subsection{Doubling and Assouad dimension}
\label{subsec:doubling-assouad}

A metric space $(Z,d)$ is called \textit{doubling} if there exists
$N\in\mathbb N$ such that every ball of radius $r>0$ can be covered by at most
$N$ balls of radius $r/2$.

We will not need the definition of Assouad dimension. We only use the standard
facts that
\[
\dim_H(Z,d)\le  \dim_A(Z,d),
\]
and that a metric space has finite Assouad dimension if and only if it is doubling; see, for example, \cite{Heinonen2001LecturesonAnalysis}. Consequently, if $\dim_H(Z,d)=\infty$, then $\dim_A(Z,d)=\infty$, and $(Z,d)$ is not doubling. For instance, see \cite{Heinonen2001LecturesonAnalysis} for more details.

\subsection{Nagata dimension and capacity dimension}
\label{subsec:nagata-capacity-dimension}

We now recall Nagata dimension. Let $\mathcal U$ be a cover of a metric space $(Z,d)$, and let $s>0$. The cover $\mathcal U$ has \textit{$s$-multiplicity at most $n+1$} if every subset of $Z$ of diameter at most $s$ meets at most
$n+1$ elements of $\mathcal U$.

The metric space $(Z,d)$ has \textit{Nagata dimension} at most $n$ if there exists a constant $c>0$ such that for every $s>0$, there is a cover $\mathcal U_s$ of $Z$ satisfying:
\begin{enumerate}
    \item $\operatorname{diam} U\le  cs$ for every $U\in\mathcal U_s$;
    \item $\mathcal U_s$ has $s$-multiplicity at most $n+1$.
\end{enumerate}
We write $\dim_N(Z,d)\le  n$. The least such $n$ is denoted $\dim_N(Z,d)$.

We will use the standard inequality
\[
\dim Z\le  \dim_N(Z,d),
\]
where $\dim Z$ denotes the covering dimension of the underlying topological
space.

The following result is the form of the Lang-Schlichenmaier \cite[Lemma~2.1]{LangSchlichenmaier2005} criterion that we will use later.

\begin{lemma}[Lang-Schlichenmaier \cite{LangSchlichenmaier2005}, Lemma~2.1]
\label{lem:lang-schlichenmaier-homeo}
Let $f:(Z,d_Z)\to(W,d_W)$ be a homeomorphism. Suppose that there exist two
homeomorphisms $\varphi,\psi:[0,\infty)\to[0,\infty)$ such that
\[
\varphi(d_Z(x,y))
\le 
d_W(f(x),f(y))
\le 
\psi(d_Z(x,y))
\]
for all $x,y\in Z$. Suppose moreover that for every $\bar c>0$, $
\sup_{s>0}
\frac{1}{s}\varphi^{-1}\bigl(\bar c\,\psi(s)\bigr)
<
\infty$.
Then $\ndim (Z,d_Z)\le  \dim_N(W,d_W)$.
\end{lemma}

\begin{remark}
The \textit{capacity dimension} $\operatorname{cdim}$, also called \textit{linearly controlled metric dimension}, is the small-scale version of Nagata dimension: one imposes the same type of linearly controlled covers, but only for sufficiently small scales. For compact metric spaces, capacity dimension and Nagata dimension agree since the small-scale condition gives the Nagata covers for all sufficiently small $s$, while for large $s$ the one-element cover handles the remaining scales. Therefore, henceforth, we will state results by referring to these both dimensions as Nagata dimension.    
\end{remark}

We recall the following theorem of Buyalo-Lebedeva \cite[Theorems~1.4,~1.5,~6.3 and~6.4]{BuyaloLebedevaPontryagin2007}, which will be used for the proof of Theorem~\ref{thm:nagata-log-boundary}.

\begin{theorem}[Buyalo-Lebedeva \cite{BuyaloLebedevaPontryagin2007}]
\label{thm:buyalo-lebedeva}
Let $X$ be a cobounded, hyperbolic, proper, geodesic space, and let $d_{\mathrm{vis}}$ be a visual metric on $\partial X$. Then
\[
    \dim_N(\partial X,d_{\mathrm{vis}})
    =
    \dim(\partial X) \quad \text{ and } \quad \asdim X
    =
    \dim(\partial X)+1
\]
where $\asdim$ denotes the asymptotic dimension. 

In particular, for a $\delta$-hyperbolic group $G$, $\ndim (\partial G,d_{\mathrm{vis}})
    =  \dim(\partial G)$ and $\asdim G = \dim(\partial G)+1$.
\end{theorem}

For the definition of asymptotic dimension, see \cite{survey:belldrashnikov_asdim}.

\subsection{Linear connectedness and uniform perfectness}
\label{subsec:linear-connectedness-uniform-perfectness-prelims}

We finally recall two metric properties.

A metric space $(Z,d)$ is called \textit{linearly connected} if there exists $L\ge 1$ such that for every $x,y\in Z$, there is a connected subset $S\subseteq Z$ with $x,y\in S$ and $\diam S\le  Ld(x,y)$.

This property is also commonly called \textit{bounded turning}; see \cite[Definition~2.7]{TukiaVaisala1980}.

A metric space $(Z,d)$ is called \textit{uniformly perfect} if there exists $\lambda>1$ such that for every $x\in Z$ and every $0<r<\diam Z$, there exists $y\in Z$ such that
\[
    \frac{r}{\lambda}
    \le 
    d(x,y)
    \le 
    r.
\]
Under this convention, finite metric spaces are not uniformly perfect, however, in the applications below, the relevant non-elementary hyperbolic group boundaries are infinite compact spaces. For more details, see \cite{Heinonen2001LecturesonAnalysis}.

\section{Metric consequences of the logarithmic gauge}
\label{S:metric-consequences}
In this section, we study various metric properties of the log metric on $\bdy X$. The results in this section
show that the log metrics and the visual metrics are in different conformal gauges with contrasting small-scale metric properties. However, despite these differences, both metrics detect numerous large-scale geometric properties alike.

Throughout this section, $X$ denotes a proper geodesic $\delta$-hyperbolic space with basepoint $\go$, $d_{\mathrm{vis}}$ denotes a visual metric on $\bdy X$ with visual parameter $\varepsilon$, and $d_{\mathrm{log}}$ denotes a log metric on $\bdy X$ with $\lbase =2$, a fixed admissible exponent $\theta>0$, basepoint $\go$ and a fixed $\kappa$.

\subsection{Hausdorff dimension and non-doubling}
\label{subsec:identity-log-to-visual}

The key analytic observation is that the identity map from the logarithmic
boundary to the visual boundary has an exponentially small modulus of
continuity.

\begin{proposition}\label{prop:log-to-visual-distortion}
There exist constants $a_1,a_2,b_1,b_2>0$ such that for all distinct $\bfa,\bfb\in \bdy X$, if $t=d_{\mathrm{log}}(\bfa,\bfb)$, then
\[
    a_1\exp(-a_2t^{-1/\theta})
    \le 
    d_{\mathrm{vis}}(\bfa,\bfb)
    \le 
    b_1\exp(-b_2t^{-1/\theta}).
\]
Moreover, the identity map $\id:(\partial X,d_{\mathrm{log}})\longrightarrow(\partial X,d_{\mathrm{vis}})$ has modulus of continuity satisfying
\[
    \omega_{\id}(t)
    \le 
    b_1\exp(-b_2t^{-1/\theta})
\]
for all sufficiently small $t>0$.
\end{proposition}

\begin{proof}
By Corollary~\ref{cor:log-metric-gromov-product}, there exist constants
$A_1,A_2>0$ such that for all distinct $\bfa,\bfb\in\bdy X$,
\[
A_1\bigl(1+(\bfa\mid\bfb)_\go\bigr)^{-\theta}
\le 
d_{\mathrm{log}}(\bfa,\bfb)
\le 
A_2\bigl(1+(\bfa\mid\bfb)_\go\bigr)^{-\theta},
\]
and there exist constants
$B_1,B_2>0$ and $\varepsilon>0$ such that
\[
    B_1e^{-\varepsilon(\bfa\mid\bfb)_\go}
    \le 
    d_{\mathrm{vis}}(\bfa,\bfb)
    \le 
    B_2e^{-\varepsilon(\bfa\mid\bfb)_\go}.
\]
Let $P=(\bfa\mid\bfb)_\go$. Suppose that $d_{\mathrm{log}}(\bfa,\bfb)= t$. 
Then the lower bound for $d_{\mathrm{log}}$ gives $A_1(1+P)^{-\theta} \le  t$. Hence,
\[
    1+P
    \ge 
    \left(\frac{A_1}{t}\right)^{1/\theta}
    \quad \implies \quad
    P
    \ge 
    A_1^{1/\theta}t^{-1/\theta}-1.
\]
Using the upper visual estimate, we obtain,
\[
\begin{aligned}
d_{\mathrm{vis}}(\bfa,\bfb)
&\le 
B_2e^{-\varepsilon P} \\
&\le 
B_2e^\varepsilon
\exp\left(-\varepsilon A_1^{1/\theta}t^{-1/\theta}\right).
\end{aligned}
\]
This gives the upper estimate after setting $b_1=B_2e^\varepsilon$ and $b_2=\varepsilon A_1^{1/\theta}$.

For the lower estimate, we use the upper bound for the log metric $ t\le  A_2(1+P)^{-\theta}$. Thus,
\[
1+P
\le 
\left(\frac{A_2}{t}\right)^{1/\theta} \quad \implies \quad
P
\le 
A_2^{1/\theta}t^{-1/\theta}-1.
\]
Using the lower visual estimate, we get
\[
\begin{aligned}
d_{\mathrm{vis}}(\bfa,\bfb)
&\,\ge\, 
B_1e^{-\varepsilon P} \,\ge\, 
B_1e^\varepsilon
\exp\left(-\varepsilon A_2^{1/\theta}t^{-1/\theta}\right).
\end{aligned}
\]
This gives the lower estimate after renaming constants.

Finally, suppose $d_{\mathrm{log}}(\bfa,\bfb)\le  t$. Then, by the upper estimate just proved,
\[
d_{\mathrm{vis}}(\bfa,\bfb)
\le 
b_1\exp\left(-b_2\,d_{\mathrm{log}}(\bfa,\bfb)^{-1/\theta}\right)
\le 
b_1\exp(-b_2t^{-1/\theta}).
\]
Taking the supremum over all pairs with
$d^X_{\go,\theta}(\bfa,\bfb)\le  t$ gives
$
\omega_{\id}(t)
\le 
b_1\exp(-b_2t^{-1/\theta})
$
for all sufficiently small $t>0$.
\end{proof}

\begin{theorem}
\label{thm:subset-positive-visual-dim-infinite-log-dim}
If $E\subseteq\partial X$ satisfies $\dim_H(E,d_{\mathrm{vis}})>0$, 
then $ \dim_H(E,d_{\mathrm{log}})=\infty$.
\end{theorem}

\begin{proof}
Consider the restriction of the identity map $g=\id|_E:
(E,d_{\mathrm{log}})
\longrightarrow
(E,d_{\mathrm{vis}})$. This map is surjective. Moreover, its modulus of continuity is bounded by the modulus of the identity on the entire boundary: $
\omega_g(t)
\le 
\omega_{\id}(t)$.
By Proposition~\ref{prop:log-to-visual-distortion}, there exist constants
$b_1,b_2>0$ such that $\omega_{\id}(t)
\le 
b_1\exp(-b_2t^{-1/\theta})$ for all sufficiently small $t>0$.
Since $\dim_H(E,d_{\mathrm{vis}})>0$, Corollary~
\ref{cor:modulus-implies-infinite-Hausdorff-dim}, applied with
$q=1/\theta$, gives $\dim_H(E,d_{\mathrm{log}})=\infty$.
\end{proof}

\begin{corollary}
\label{thm:positive-visual-dim-implies-infinite-log-dim}
Suppose that $\dim_H(\bdy X,d_{\mathrm{vis}})>0$.
Then $
\dim_H(\bdy X,d_{\mathrm{log}})=\infty$, $\dim_A(\bdy X,d_{\mathrm{log}})=\infty$, and $(\bdy X,d_{\mathrm{log}})$ is not doubling.
\end{corollary}
\begin{proof}
Apply Theorem~
\ref{thm:subset-positive-visual-dim-infinite-log-dim}
with $E=\partial X$. The remaining conclusions follow from the fact $
\dim_H Z\le \dim_A Z$ for any metric space $Z$ and the equivalence between finite Assouad dimension and doubling as discussed in Section~\ref{subsec:doubling-assouad}.
\end{proof}

\begin{corollary}
\label{cor:log-continua-infinite-dimensional}
Every nondegenerate connected subset
    $K\subseteq\partial_\kappa G$ satisfies
    \[
    \dim_H(K,d_{\mathrm{log}})=\infty.
    \]
    Moreover, $(\bdy X, \dlog)$ contains no nonconstant
    rectifiable curves. In particular, every Lipschitz map from an interval into $(\partial G,d_{\mathrm{log}})$ is constant.
\end{corollary}
\begin{proof}
    Let $K\subseteq \bdy X$ be a non-degenerate connected subset. Then there exist two points $\bfa, \bfb \in K$ so that $d_{\mathrm{vis}}(\bfa, \bfb) >0$. Fix one of the points, say $\bfa \in K$, and define $f_\bfa: K\to \R$ as $f(\bfx) = d_{\mathrm{vis}}(\bfa, \bfx)$. This is $1$-Lipschitz map
    and $$[0, d_{\mathrm{vis}}(\bfa, \bfb)] \subseteq f(K).$$
    Thus, $f(K)$ has Hausdorff dimension at least $1$. Since Lipschitz maps cannot increase Hausdorff dimension,
    \[
        1\,\le \,\dim_H f(K) \,\le\,\dim_H (K, d_{\mathrm{vis}}).
    \]
    By Theorem~\ref{thm:subset-positive-visual-dim-infinite-log-dim}, $\dim_H(K, d_{\mathrm{log}}) = \infty$. 
    
    Now suppose that $ \gamma:[a,b]\longrightarrow(\partial X,d_{\mathrm{log}}) $  is a nonconstant rectifiable curve, and let $E=\gamma([a,b])$. Then $E$ is a nondegenerate connected subset of $\partial X$, so,  $ \dim_H(E,d_{\mathrm{log}})=\infty$.  On the other hand, rectifiability implies \[ \mathcal H^1(E,d_{\mathrm{log}}) \le  \operatorname{length}_{d_{\mathrm{log}}}(\gamma) < \infty, \] 
    where $\mathcal H^1(E,d_{\mathrm{log}})$ is the Hausdorff 1-measure, and therefore $\dim_H(E,d_{\mathrm{log}})\le 1$. This is a contradiction. Hence $(\partial X,d_{\mathrm{log}})$ contains no nonconstant rectifiable curves.
\end{proof}

When $G$ is a word-hyperbolic group, we call $G$ \textit{elementary} if it is finite or virtually cyclic, and \textit{non-elementary} otherwise. Equivalently, if $G$ is finite then $\partial G=\emptyset$, if $G$ is infinite virtually cyclic then $\partial G$ has exactly two points, and if $G$ is non-elementary then $\partial G$ is an infinite perfect compact metrizable space; see, for example, \cite[Theorem~2.8]{survey:kapovichbenakli}. We use the convention
$\dim_H\emptyset=0$.

\begin{corollary}
\label{cor:group-hausdorff-assouad-nondoubling}
Let $G$ be a word-hyperbolic group.  If $G$ is elementary, then
$
\dim_H(\bdy G,d_{\mathrm{log}})=0$. If $G$ is non-elementary, then
\begin{enumerate}
    \item $\dim_H(\bdy G,d_{\mathrm{log}})=\infty$, $\dim_A(\bdy G,d_{\mathrm{log}})=\infty$,
    \item $(\bdy G,d_{\mathrm{log}})$ is not doubling, and
    \item $(\partial G,d_{\mathrm{log}})$ is nowhere locally
doubling.
\end{enumerate}

In particular, $(\bdy G,d_{\mathrm{log}})$ is not quasisymmetrically equivalent to
$(\partial G,d_{\mathrm{vis}})$ for any visual metric $d_{\mathrm{vis}}$ on a non-elementary group $G$.
\end{corollary}

\begin{proof}
If $G$ is finite, then $\partial G=\emptyset$. If $G$ is infinite virtually
cyclic, then $\partial G$ has exactly two points. Hence, we have $\dim_H(\bdy G,d_{\mathrm{log}})=0$.

Now suppose that $G$ is non-elementary. Fix a visual metric $\dvis$ on $\partial G$ with parameter $a>1$. Let $e(G)$ denote the critical exponent of the action of $G$ on its Cayley graph, and let $e_a(G) = e(G)/\log{a}$. By \cite[Corollaries~5.5 and~7.6]{measure:Coornaert93}, $e_a(G) >0$ and $0< \mathcal H^{e_a(G)}_{d_a}(\partial G) < \infty$. It follows that $\dim_H(\partial G,d_{\mathrm{vis}})=e_a(G)>0$. Corollary~\ref{thm:positive-visual-dim-implies-infinite-log-dim}
therefore gives $\dim_H(\partial G,d_{\mathrm{log}})=\infty$.

Moreover, let $U$ be a nonempty open set equipped with log metric. It contains a nonempty visual ball $B(\bfa, r) \subseteq U$ since the log metric and the visual metric induce the same topology. Since $\dim_H (B(\bfa, r), d_{\mathrm{vis}}) >0 $, hence by Theorem~\ref{thm:positive-visual-dim-implies-infinite-log-dim}, $\dim_H(B(\bfa, r), d_{\mathrm{log}}) = \infty$. By monotonicity, $\dim_H(U, d_{\mathrm{log}}) = \infty$.

Since Hausdorff dimension is bounded above by Assouad dimension, we also obtain $\dim_A(\bdy G,d_{\mathrm{log}})=\infty$. A metric space is doubling if and only if it has finite Assouad dimension, hence $(\b G,d_{dy\mathrm{log}})$ is not doubling and $(U, d_{\mathrm{log}})$ is not doubling. Thus, no nonempty open subset of $(\bdy X, d_{\mathrm{log}})$ is doubling.

Finally, the boundary of a hyperbolic group equipped with any visual metric is doubling, and doubling is preserved under quasisymmetric homeomorphisms. Therefore, $(\bdy G,d_{\mathrm{log}})$ is not quasisymmetrically equivalent to
$(\partial G,d_{\mathrm{vis}})$ for any visual metric $d_{\mathrm{vis}}$. In particular, the logarithmic and visual metrics belong to distinct
conformal gauges.
\end{proof}

Thus far, we have observed that the small-scale properties are altered significantly under the change of metric from the visual metric to log metric. In the following subsections, we show that several large-scale properties of hyperbolic spaces are still detected by this change of metrics.

\subsection{Nagata dimension and asymptotic dimension}
\label{subsec:nagata-asdim-results}
We now prove that Nagata dimension behaves very differently from Hausdorff and
Assouad dimension. Although the logarithmic boundary has infinite Hausdorff and
Assouad dimension in the non-elementary case, its Nagata dimension is still coincides with the topological dimension of the boundary. 

\begin{proposition}
\label{prop:nagata-log-leq-visual}
Let $X$ be a proper geodesic hyperbolic space. Then
\[
\dim_N(\bdy X,d_{\mathrm{log}})
\le 
\dim_N(\bdy X,d_{\mathrm{vis}}).
\]
\end{proposition}

\begin{proof}
We apply Lemma~\ref{lem:lang-schlichenmaier-homeo} to the identity map $\id:(\bdy X,d_{\mathrm{log}})
\longrightarrow
(\bdy X,d_{\mathrm{vis}})$. Let $q =1/\theta$. By Proposition~\ref{prop:log-to-visual-distortion}, there exist constants
$a_1,a_2,b_1,b_2>0$ such that, for all distinct $\bfa,\bfb$, if
$t=d_{\mathrm{log}}(\bfa,\bfb)$, then
\[
    a_1\exp(-a_2t^{-q})
    \le 
    d_{\mathrm{vis}}(\bfa,\bfb)
    \le 
    b_1\exp(-b_2t^{-q}).
\]

Since the boundary is compact, the log metric is bounded. Choose $T>0$ such that $d_{\mathrm{log}}(\bfa,\bfb)<T$
for all $\bfa,\bfb\in\bdy X$. Define increasing homeomorphisms $\varphi,\psi:[0,\infty)\to[0,\infty)$ as follows:
\begin{itemize}
    \item define $\varphi(0)=\psi(0)=0$,
    \item for $0<t\le  T$, define 
    $$\varphi(t)=a_1\exp(-a_2t^{-q}) \quad \text{ and }\quad \psi(t)=c_1\exp(-c_2t^{-q}),$$
    \item for $t\ge T$, extend linearly $$\varphi(t)=\varphi(T)+(t-T) \quad \text{ and } \quad \psi(t)=\psi(T)+(t-T).$$
\end{itemize}

Then $\varphi$ and $\psi$ are increasing homeomorphisms of
$[0,\infty)$ onto $[0,\infty)$. Moreover, for all
$\bfa,\bfb\in\pka X$, since all realized values of $d_{\mathrm{log}}(\bfa,\bfb)$ lie in $[0,T)$,
\[
\varphi(d^X_{\go,\theta}(\bfa,\bfb))
\le 
d_{\mathrm{vis}}(\bfa,\bfb)
\le 
\psi(d^X_{\go,\theta}(\bfa,\bfb)).
\]

It remains to verify the quantitative hypothesis in
Lemma~\ref{lem:lang-schlichenmaier-homeo}. Fix $\bar c>0$. We claim that
\[
\sup_{s>0}
\frac{1}{s}
\varphi^{-1}\bigl(\bar c\,\psi(s)\bigr)
<\infty.
\]

Since $\psi(s)\to0$ as $s\downarrow0$, we may choose $s_0>0$ sufficiently
small so that $s_0<T$ and, for every $0<s<s_0$, the number
$\bar c\,\psi(s)$ lies in the range of the first branch of $\varphi$. For
$0<s<s_0$, we compute
\[
\bar c\,\psi(s)
=
\bar c b_1\exp(-b_2s^{-q}), \,\,\,\,\text{and}
\]
\[
\varphi^{-1}(u)
=
\left(
\frac{a_2}{\log(a_1/u)}
\right)^{1/q}.
\]
Hence,
\[
\begin{aligned}
\varphi^{-1}\bigl(\bar c\,\psi(s)\bigr)
&=
\left(
\frac{a_2}
{\log\left(a_1/(\bar c b_1\exp(-b_2s^{-q}))\right)}
\right)^{1/q}  \\
&=
\left(
\frac{a_2}
{b_2s^{-q}+\log(a_1/(\bar c b_1))}
\right)^{1/q}.
\end{aligned}
\]
Since $b_2s^{-q}\to\infty$ as $s\downarrow0$, after decreasing $s_0$ if
necessary we may assume that
\[
b_2s^{-q}+\log(a_1/(\bar c b_1))
\ge 
\frac{b_2}{2}s^{-q}
\]
for every $0<s<s_0$. Thus,
\[
\varphi^{-1}\bigl(\bar c\,\psi(s)\bigr)
\le 
\left(\frac{2a_2}{b_2}\right)^{1/q}s.
\]
Therefore $\tfrac{1}{s}\varphi^{-1}\bigl(\bar c\,\psi(s)\bigr)$
is bounded for $0<s<s_0$.

On the compact interval $[s_0,T]$, the function
\[
s\mapsto
\frac{1}{s}\varphi^{-1}\bigl(\bar c\,\psi(s)\bigr)
\]
is continuous, and hence bounded.

Finally, for $s\ge  T$, the linear extension gives
\[
\varphi^{-1}(u)\le  T+u
\]
for all $u\ge 0$, while
\[
\psi(s)=\psi(T)+(s-T)\le  s+\psi(T).
\]
Therefore,
\[
\frac{1}{s}\varphi^{-1}\bigl(\bar c\,\psi(s)\bigr)
\le 
\frac{T}{s}
+
\bar c\frac{\psi(s)}{s}
\le 
1+\bar c+\frac{\bar c\,\psi(T)}{T}.
\]
This proves the desired supremum bound. All hypotheses of Lemma~\ref{lem:lang-schlichenmaier-homeo} are now satisfied,
so
\[
\dim_N(\bdy X,d_{\mathrm{log}})
\le 
\dim_N(\bdy X,d_{\mathrm{vis}}).
\]
\end{proof}

\begin{theorem}
\label{thm:nagata-log-boundary}
Let $G$ be a non-elementary hyperbolic group. Then
\[
\dim_N(\bdy G,d_{\mathrm{log}})
=
\dim_N(\bdy G,d_{\mathrm{vis}})
=
\dim(\bdy G).
\]
Moreover,
$
\asdim G
=
\dim_N(\bdy G,d_{\mathrm{log}})+1.
$
\end{theorem}

\begin{proof}
Let $d_{\mathrm{vis}}$ be a visual metric on $\partial G$. By
Proposition~\ref{prop:nagata-log-leq-visual},
\[
\dim_N(\bdy G,d_{\mathrm{log}})
\le 
\dim_N(\partial G,d_{\mathrm{vis}}).
\]
By the Buyalo-Lebedeva theorem~\ref{thm:buyalo-lebedeva}, $
\dim_N(\partial G,d_{\mathrm{vis}})
=
\dim(\partial G)$.
Hence,
\[
\dim_N(\bdy G,d_{\mathrm{log}})
\le 
\dim(\partial G).
\]
On the other hand, topological dimension is bounded above by Nagata dimension:
\[
\dim(\bdy G)
\le 
\dim_N(\bdy G,d_{\mathrm{log}}).
\]
Thus, $\dim_N(\bdy G,d_{\mathrm{log}})
=
\dim(\bdy G)$. Finally, by Buyalo-Lebedeva's result
\[
\asdim G=\dim(\partial G)+1
\]
which implies
\[
\asdim G
=
\dim_N(\bdy G,d_{\mathrm{log}})+1.
\]
\end{proof}

\subsection{Linear connectedness and uniform perfectness}
\label{subsec:lc-up-results}

We finish by showing that some metric properties, such as uniform perfectness and linear connectedness, survive the logarithmic change of scale even when log metrics and visual metrics are not in conformal gauges of each other. 

Throughout this section, fix constants $\epsilon>0$ and $A\ge  1$ such that for all $\bfa,\bfb\in\partial X$,
\begin{equation}\label{eq:visual_estimates}    
    A^{-1}e^{-\epsilon(\bfa|\bfb)_{\go}}
    \le 
    d_{\mathrm{vis}}(\bfa,\bfb)
    \le 
    A e^{-\epsilon(\bfa|\bfb)_{\go}}.
\end{equation}

Also fix $B\ge  1$ such that
\begin{equation}\label{eq:log_estimates}
    B^{-1}\bigl(1+(\bfa|\bfb)_{\go}\bigr)^{-\theta}
    \le 
    d_{\mathrm{log}}(\bfa,\bfb)
    \le 
    B\bigl(1+(\bfa|\bfb)_{\go}\bigr)^{-\theta}.
\end{equation}

\begin{proposition}
\label{prop:visual-lc-implies-log-lc}
Let $X$ be a proper geodesic hyperbolic space. If
$(\partial X,d_{\mathrm{vis}})$ is linearly connected, then
$(\partial X,d_{\mathrm{log}})$ is linearly connected.
\end{proposition}
\begin{proof}
Let $L\ge  1$ be a linear connectedness constant for $(\partial X,d_{\mathrm{vis}})$. Let $\bfa,\bfb\in \pka X$. If $\bfa=\bfb$, there is nothing to prove. Assume
$\bfa\neq \bfb$. By linear connectedness of $(\bdy X,d_{\mathrm{vis}})$, there
exists a connected set $E\subset \bdy X$ containing $\bfa$ and $\bfb$ such that
\[
    \diam_{\mathrm{vis}}(E)
    \le 
    L d_{\mathrm{vis}}(\bfa,\bfb).
\]
Since the two metrics induce the same topology, $E$ is also connected with respect to $d_{\mathrm{log}}$. Now let $\alpha,\beta\in E$. Then
\[
    d_{\mathrm{vis}}(\alpha,\beta)
    \le 
    \diam_{\mathrm{vis}}(E)
    \le 
    L d_{\mathrm{vis}}(\bfa,\bfb).
\]
Using the estimates in \eqref{eq:visual_estimates}, we obtain
\[
    A^{-1}e^{-\epsilon(\alpha|\beta)_{\go}}
    \le 
    d_{\mathrm{vis}}(\alpha,\beta)
    \le 
    L A e^{-\epsilon(\bfa|\bfb)_{\go}}.
\]
Taking logarithms
\[
    (\alpha|\beta)_{\go}
    \ge 
    (\bfa|\bfb)_{\go}
    -
    C, \quad \text{ where }C=\frac{1}{\epsilon}\log(LA^2).
\]
Equivalently,
\[
    \inf_{\alpha,\beta\in E}(\alpha|\beta)_{\go}
    \ge 
    (\bfa|\bfb)_{\go}-C.
\]
Since $(\alpha|\beta)_{\go}\ge  0$, the previous inequality implies
\[
    1+(\bfa|\bfb)_{\go}
    \le 
    (1+C)\bigl(1+(\alpha|\beta)_{\go}\bigr)
\]
for every $\alpha,\beta\in E$. 
Using the comparison with $d_{\mathrm{log}}$, we obtain
\[
\begin{aligned}
    d_{\mathrm{log}}(\alpha,\beta)
    &\le 
    B\bigl(1+(\alpha|\beta)_{\go}\bigr)^{-\theta}  \\
    &\le 
    B(1+C)^\theta
    \bigl(1+(\bfa|\bfb)_{\go}\bigr)^{-\theta}       \\
    &\le 
    B^2(1+C)^\theta d_{\mathrm{log}}(\bfa,\bfb).
\end{aligned}
\]
Taking the supremum over all $\alpha,\beta\in E$ gives
\[
    \diam_{\log}(E)
    \le 
    B^2(1+C)^\theta d_{\mathrm{log}}(\bfa,\bfb).
\]
Thus $E$ is a connected subset of $(\partial X,d_{\mathrm{log}})$ containing $\bfa,\bfb$
with
\[
    \diam_{\log}(E)
    \le 
    L' d_{\mathrm{log}}(\bfa,\bfb),  \quad \text{ where } L'=B^2(1+C)^\theta.
\]
Hence $(\partial X,d_{\mathrm{log}})$ is linearly connected.
\end{proof}
\begin{corollary}[Characterization of one-ended groups]
\label{cor:lc-hyperbolic-groups}
Let $G$ be a non-elementary hyperbolic group. Then the following are equivalent:

\begin{enumerate}
    \item $(\partial G,d_{\mathrm{vis}})$ is linearly connected
    \item $(\partial G,d_{\mathrm{log}})$ is linearly connected
    \item  $G$ is one-ended.
\end{enumerate}
\end{corollary}
\begin{proof}
    If $(\partial G,d_{\mathrm{vis}})$ is linearly connected, then
Proposition~\ref{prop:visual-lc-implies-log-lc} implies that
$(\pka G,d^G_{e,\theta})$ is linearly connected. Conversely, suppose that $(\partial G, d_{\mathrm{log}})$ is linearly connected. Then $\partial G$ is connected. For a hyperbolic group, connectedness of the boundary is equivalent to being one-ended. By Bonk-Kleiner \cite[Proposition~4]{BonkKleiner:llc_one_ended}, the visual boundary of a one-ended hyperbolic group is linearly connected. Hence
$(\partial G,d_{\mathrm{vis}})$ is linearly connected.

\end{proof}
\begin{proposition}
\label{prop:visual-up-implies-log-up}
    Let $X$ be a proper geodesic hyperbolic space. If $(\vis X,d_{\mathrm{vis}})$ is uniformly perfect, then $(\pka X,d_{\mathrm{log}})$ is uniformly perfect. Consequently, if $G$ is non-elementary hyperbolic group, then $(\bdy G, d_{\mathrm{log}})$ is uniformly perfect. 
\end{proposition}

\begin{proof}
Let $\sf{c}_{\mathrm{vis}}\in(0,1)$ be a uniform perfectness constant for $(\partial X,d_{\mathrm{vis}})$. 
We prove that there exists $\sf{c}_{\mathrm{log}}\in(0,1)$ such that for every $\bfa \in \bdy X$ and for every $0<R<\diam_{\log}(\bdy X)$, there exists  $\bfb \in \bdy X$ so that
\[
    {\sf{c}}_{\mathrm{log}} R \,\le\,  d_{\mathrm{log}}(\bfa, \bfb) \,\le\, R.
\]
Let $0< R \le \diam_{\mathrm{log}} \bdy X$. Let $P = (B/R)^{1/\theta}-1$ and $s = A^{-1}e^{-\varepsilon P}$. Choose $P_0 \ge 0$ so that $$s_0 : = A^{-1}e^{-\varepsilon P_0} < \diam_{\mathrm{vis}}(\bdy X).$$
Let $R_0 = B(1+P_0)^{-\theta}$.

\noindent Case 1) If $0<R\le R_0$ then $P \ge P_0$ and $s \le s_0 < \diam_{\mathrm{vis}}(\bdy X)$ and hence, by unifrom perfectness of $(\bdy X, d_{\mathrm{vis}})$, there exists $\bfb \in \bdy X$ so that
\[
    {\sf{c}}_{\mathrm{vis}} s \,\le\,  d_{\mathrm{vis}}(\bfa, \bfb) \,\le\, s \,\implies\, {\sf{c}}_{\mathrm{vis}}A^{-1}e^{-\varepsilon P}   \,\le\,  d_{\mathrm{vis}}(\bfa, \bfb) \,\le\, A^{-1}e^{-\varepsilon P}.
\]
The upper bound and \eqref{eq:visual_estimates} imply that $(\bfa \mid \bfb)_{\go} \ge P$. Thus by \eqref{eq:log_estimates}, we obtain $$d_{\mathrm{log}}(\bfa, \bfb)\le B(1+P)^{-\theta} = R.$$
Similarly, from the lower bound, we obtain $(\bfa \mid \bfb)_{\go} \le P+C$ where $C = \frac{1}{\varepsilon} \log\Bigl(\frac{A^2}{{\sf c}_{\mathrm{vis}}}\Bigr)$. Thus
$P \le (\bfa\mid\bfb)_{\go} \le P+C$ and 
\[
\begin{aligned}
    d_{\mathrm{log}}(\bfa, \bfb)
    &\ge 
    B^{-1}(1+P+C)^{-\theta}  \\
    &=
    B^{-1}\left(\left(\frac BR\right)^{1/\theta}+C\right)^{-\theta} \\
    &=
    B^{-2}R
    \left(1+C\left(\frac RB\right)^{1/\theta}\right)^{-\theta}.
\end{aligned}
\]
Since $R<\operatorname{diam}_{\log}(\partial X)\le  B$, we have
$(R/B)^{1/\theta}\le  1$. Hence
\[
    d_{\mathrm{log}}(\bfa, \bfb)
    \ge 
    B^{-2}(1+C)^{-\theta}R \quad \text{ where } C = \frac{1}{\varepsilon} \log\left(\frac{A^2}{{\sf c}_{\mathrm{vis}}}\right).
\]
Case 2) Let $R_0<R<\operatorname{diam}_{\log}(\partial X)$. Applying the previous visual annulus argument for a fixed scale $s_0 = A^{-1 }e^{-\varepsilon P_0}$, we obtain that there exists $\bfb\in \bdy X$ so that 
$$P_0 \le (\bfa\mid\bfb)_{\go} \le P_0+C \quad \text{ where }C = \frac{1}{\varepsilon} \log\left(\frac{A^2}{{\sf c}_{\mathrm{vis}}}\right).$$ 
Therefore
\[
    d_{\mathrm{log}}(\bfa, \bfb)
    \le 
    B(1+P_0)^{-\theta}
    =
    R_0
    <
    R,
\]
and since $R<\operatorname{diam}_{\log}(\partial X)\le  B$, we get
\[
    d_{\mathrm{log}}(\bfa, \bfb)
    \ge 
    B^{-1}(1+P_0+C)^{-\theta} \ge  
    B^{-2}(1+P_0+C)^{-\theta}R.
\]
Choosing ${\sf c}_{\mathrm{log}} = \min\left\{
B^{-2}(1+C)^{-\theta}, B^{-2}(1+P_0+C)^{-\theta}\right\}$ proves the claim.
\end{proof}

\appendix
\section{Homeomorphism between $\pka X$ and $\vis X$} \label{app:homeo}
In this section, we prove that for a proper, geodesic, $\delta$-hyperbolic space $X$, $\pka X$ and $\vis X$ are homeomorphic.

Although this identification follows from
\cite[Proposition~7.3]{CashenMackay19Lmetrizable-contracting},
\cite[Theorem~1.1]{HeEquivalentTopologies}, and
\cite[Proposition~4.13]{QRT2},
we include a direct proof in order to make explicit the comparison
between the Gromov-product neighborhoods on $\partial_\infty X$
and the $\kappa$-Morse neighborhoods on $\partial_\kappa X$.

We assume that all quasi-geodesics in $X$ are based at $\go$.
Define $f:\vis X\longrightarrow \pka X$ as follows. If $\bfa\in\vis X$ and $\alpha$ is a geodesic ray based at $\go$ representing $\bfa$, let $f(\bfa):=[\alpha]_\kappa$.

The next Lemma says that geodesic representatives of distinct boundary points diverge linearly after their common initial fellow-traveling time.

\begin{lemma}
\label{lem:linear-divergence}
Let $\bfa,\bfb\in \vis X$ be distinct boundary points, and let
$\alpha_0,\beta_0:[0,\infty)\to X$ be geodesic representatives of
$\bfa,\bfb$, respectively, based at $\go$. Then for every $t\ge  0$,
\[
d\bigl(\alpha_0(t),\beta_0(t)\bigr)
\ge 
2t-2(\bfa\mid\bfb)_\go-4\delta.
\]
In particular,
\[
\liminf_{t\to\infty}
\frac{d\bigl(\alpha_0(t),\beta_0(t)\bigr)}{t}
\ge  2.
\]
\end{lemma}

\begin{proof}
Let $P=(\bfa\mid\bfb)_\go$.
Fix $t\ge  0$. We first show that
$$
(\alpha_0(t)\mid\beta_0(t))_\go\le  P+2\delta.
$$
Let $u,v\ge  t$. Applying $\delta$-hyperbolicity twice gives
\[
\begin{aligned}
(\alpha_0(u)\mid\beta_0(v))_\go
&\ge 
\min\{
(\alpha_0(u)\mid\alpha_0(t))_\go,
(\alpha_0(t)\mid\beta_0(t))_\go,
(\beta_0(t)\mid\beta_0(v))_\go
\}
-2\delta.
\end{aligned}
\]
Since $\alpha_0$ and $\beta_0$ are geodesic rays based at $\go$,
\[
(\alpha_0(u)\mid\alpha_0(t))_\go=t,
\qquad
(\beta_0(t)\mid\beta_0(v))_\go=t,
\qquad
(\alpha_0(t)\mid\beta_0(t))_\go\le  t.
\]
Therefore,
\[
(\alpha_0(u)\mid\beta_0(v))_\go
\ge 
(\alpha_0(t)\mid\beta_0(t))_\go-2\delta.
\]
Taking the liminf as $u,v\to\infty$ and by \eqref{eq:BH_upper_lower_bound_P}, we obtain
\[
P \,\,\ge \,\, \liminf_{u,v\to\infty}
(\alpha_0(u)\mid\beta_0(v))_\go
\,\,\ge \,\,
(\alpha_0(t)\mid\beta_0(t))_\go-2\delta.
\]
Now,
\[
d\bigl(\alpha_0(t),\beta_0(t)\bigr)
=
2t-2(\alpha_0(t)\mid\beta_0(t))_\go.
\]
Using the estimate above, we obtain
\[
d\bigl(\alpha_0(t),\beta_0(t)\bigr)
\ge 
2t-2P-4\delta
=
2t-2(\bfa\mid\bfb)_\go-4\delta.
\]
Dividing by $t$ and letting $t\to\infty$ gives
\[
\liminf_{t\to\infty}
\frac{d\bigl(\alpha_0(t),\beta_0(t)\bigr)}{t}
\ge  2.
\]
\end{proof}

\begin{proposition}\label{prop:bijectivity-visual-kappa_appendix} 
$f$ is well-defined and bijective. 
\end{proposition}

\begin{proof}
    We first check that the map is well-defined. 
    Let $\bfa\in\partial_\infty X$, and let $\alpha$ and $\alpha'$ be two geodesic rays based at $\go$ representing $\bfa$. Their images have finite Hausdorff distance. Hence, there exists $D\ge 0$ such that, for every $t\ge 0$, one can find $s\ge 0$ satisfying
    \[
    d\bigl(\alpha(t),\alpha'(s)\bigr)\le  D.
    \]
    Since both rays are based at $\go$ and parametrized by arclength, $|t-s|\le  D$. Therefore,
    \[
        d\bigl(\alpha(t),\alpha'(t)\bigr)
        \le 
        d\bigl(\alpha(t),\alpha'(s)\bigr)+|s-t|
        \le 2D,
    \]
    and consequently,
    \[
        \frac{d(\alpha(t),\alpha'(t))}{t}
        \le \frac{2D}{t}\longrightarrow0.
    \]
    Thus, $\alpha$ and $\alpha'$ sublinearly track each other, so
    $[\alpha]_\kappa=[\alpha']_\kappa$. Hence $f(\bfa)$ is independent of
    the choice of geodesic representative.

    We now prove that $f$ is injective. Let $\bfa,\bfb\in \partial_\infty X$ and
    suppose that $ f(\bfa)=f(\bfb)$. Let $\alpha$ and $\beta$ be geodesic rays based at 
    $\go$ representing $\bfa$ and $\bfb$, respectively. By the definition of $f$, we have
    $f(\bfa)=[\alpha]_\kappa$ and $f(\bfb)=[\beta]_\kappa$. Hence 
    $[\alpha]_\kappa=[\beta]_\kappa$.
    By the definition of the equivalence relation on $\partial_\kappa X$, this
    means that $\alpha$ and $\beta$ sublinearly track each other, i.e.
    \[
            \lim_{t\to\infty}
            \frac{d_X(\alpha(t),\beta(t))}{t}=0.
    \]

    We claim that $\bfa=\bfb$. Suppose, for contradiction, that $\bfa\ne\bfb$. Then their Gromov product is finite: $P:=(\bfa|\bfb)_o<\infty$
    By Lemma~\ref{lem:linear-divergence}, for every $t\ge 0$ we have $d_X(\alpha(t),\beta(t))
            \ge
            2t-2P-4\delta$. Therefore,
    \[
            \liminf_{t\to\infty}
            \frac{d_X(\alpha(t),\beta(t))}{t}
            \ge 2
    \]
    contradicting the sublinear traveling. Hence, $\bfa = \bfb$. So $f$ is injective.

    We next prove that $f$ is surjective. Let $\bfx\in \pka X$. By Lemma~\ref{lem:qrt242}, there exists a geodesic ray $\alpha:[0,\infty)\to X$ based at $\go$ such that $\bfx=[\alpha]_\kappa$. Let $\bfa=[\alpha]_\infty\in\partial_\infty X$ be the visual boundary point represented by $\alpha$. Then, by definition of $f$, $f(\bfa)=[\alpha]_\kappa=\bfx$. Hence $f$ is surjective.
\end{proof}

\begin{proposition}\label{prop:visual-kappa-homeomorphism_appendix} 
    $f$ is a homeomorphism.  
\end{proposition}

\begin{proof}
The fact that $f$ is well-defined and bijective follows from  Proposition~\ref{prop:bijectivity-visual-kappa_appendix}. It remains to
compare the two topologies.

For $\bfa\in\vis X$ and $T>0$, let
\[
\mathcal V(\bfa,T)
:=
\bigl\{\bfb\in\vis X:(\bfa\mid\bfb)_{\go}>T\bigr\}.
\]
 The sets
$\mathcal V(\bfa,T)$ form the standard neighborhood basis at $\bfa$ in
the visual boundary. Fix a geodesic representative
$\alpha_0$ of $\bfa$. The sets
\[
\calU_\kappa(\alpha_0,R),
\qquad R>0,
\]
form a neighborhood basis at $f(\bfa)$ in $\pka X$ as in Proposition~4.7 \cite{QRT2}. We show that these
two neighborhood bases generate the same topology under $f$.

First, we claim that for every $R>0$,
\begin{equation}\label{eq:visual-neighborhood-into-kappa}
f\bigl(\mathcal V(\bfa,R)\bigr)
\subseteq
\calU_\kappa(\alpha_0,R).
\end{equation}
Let $\bfb\in\mathcal V(\bfa,R)$, and denote
$
P=(\bfa\mid\bfb)_{\go}>R
$. The case $\bfb=\bfa$ is immediate, so suppose that
$\bfb\neq\bfa$. Let $\beta_0$ be a geodesic representative of
$\bfb$, and let $\beta$ be any $(\qq,\sQ)$-quasi-geodesic
representative of $f(\bfb)$ satisfying
\[
\sM(\qq,\sQ)<\frac{R}{24\kappa(R)}.
\]
We must prove that $
\beta|_R
\subseteq
\calN_\kappa\bigl(\alpha_0,\sM(\qq,\sQ)\bigr)$.

Indeed, let $z\in\beta|_R$. Then $\|z\|\le  R$. Since $\beta$ and
$\beta_0$ represent the same boundary point, stability of
quasi-geodesics provides $t\ge 0$ such that
\[
d\bigl(z,\beta_0(t)\bigr)\le  H(\delta,\qq,\sQ)
\]
as shown in Figure~\ref{fig:prop33}. For brevity, we write $H=H(\delta,\qq,\sQ)$. Since $\beta_0$ is a
geodesic ray based at $\go$,
\[
t
\le 
\|z\|+H
\le 
R+H
<
P+H.
\]
Let $s=\min\{t,P\}$, then utilizing Lemma~\ref{lem:GP_basic} and the calculations in Proposition~\ref{prop:relationship_p_r}, we obtain
\[
\begin{aligned}
d\bigl(\beta_0(t),\alpha_0(s)\bigr)
&\le 
t+s+10\delta-2\min\{t,s,P\} \\
&\le 
H+10\delta.
\end{aligned}
\]
Therefore,
\[
\begin{aligned}
d(z,\alpha_0)
&\le 
d\bigl(z,\beta_0(t)\bigr)
+
d\bigl(\beta_0(t),\alpha_0(s)\bigr)\\
&\le 
\sM(\qq,\sQ) \, \le 
\sM(\qq,\sQ)\kappa(\|z\|),
\end{aligned}
\]
where the second inequality follows from the choice of the common strong Morse gauge and the last uses $\kappa\ge 1$. Thus, $z\in \calN_\kappa\bigl(\alpha_0,\sM(\qq,\sQ)\bigr)$.
Since $z$ and $\beta$ were arbitrary, this proves
\eqref{eq:visual-neighborhood-into-kappa}.

Conversely, fix $T>0$. Let $M_0:=\sM(1,0)$. Since $\kappa$ is sublinear, we may choose $R>0$ sufficiently large
that
\begin{equation}\label{eq:choice-R-homeomorphism}
M_0<\frac{R}{24\kappa(R)}
\qquad\text{and}\qquad
R-M_0\kappa(R)-2\delta>T.
\end{equation}
We claim that $\calU_\kappa(\alpha_0,R)
\subseteq
f\bigl(\mathcal V(\bfa,T)\bigr)$.

Let $\bfy\in\calU_\kappa(\alpha_0,R)$. Since $f$ is bijective, we can write
$\bfy=f(\bfb)$ for a unique $\bfb\in\vis X$. Again, the case
$\bfb=\bfa$ is immediate. Otherwise, let $\beta_0$ be a geodesic
representative of $\bfb$. Since $\beta_0$ is a $(1,0)$-quasi-geodesic,
the first inequality in \eqref{eq:choice-R-homeomorphism} and the
definition of $\calU_\kappa(\alpha_0,R)$ give $\beta_0|_R
\subseteq
\calN_\kappa(\alpha_0,M_0)$.
In particular,
\[
d\bigl(\beta_0(R),\alpha_0\bigr)
\le 
M_0\kappa(R).
\]
Lemma~\ref{lem:gp_basic_2} now yields
\[
(\bfa\mid\bfb)_{\go}
\ge 
R-M_0\kappa(R)-2\delta
>
T.
\]
Therefore $\bfb\in\mathcal V(\bfa,T)$, proving
$\calU_\kappa(\alpha_0,R)
\subseteq
f\bigl(\mathcal V(\bfa,T)\bigr)$.

The containment
\eqref{eq:visual-neighborhood-into-kappa} shows that $f$ is continuous,
while $\calU_\kappa(\alpha_0,R)
\subseteq
f\bigl(\mathcal V(\bfa,T)\bigr)$ shows that $f^{-1}$ is
continuous. Hence $f$ is a homeomorphism.
\end{proof}

\bibliographystyle{acm}
\bibliography{references}
\end{document}